\documentclass[12pt,psfig,a4paper]{article}
\usepackage{amsthm,amssymb,amsmath,amsfonts,mathrsfs,amscd}
\usepackage{graphicx,epsfig,latexsym}
\usepackage{cite,calc,xcolor,subfigmat,dsfont}
\usepackage[%pagebackref=true,
colorlinks,linkcolor=blue,citecolor=magenta]{hyperref}
\newcommand{\LST}{\mathscr{L}}
\newcommand{\PM}{\mathscr{P}}
\newcommand{\Sl}{\mathfrak{sl}_2}
\newcommand{\R}{\mathscr{R}}
\newcommand{\RC}{\mathcal{R}}
\newcommand{\C}{\mathcal{C}}

\newcommand{\N}{\mathbb{N}}
\newcommand{\Le}{\left}
\newcommand{\Ri}{\right}
\newcommand{\Span}{{\rm span}}
\newcommand{\ad}{{\rm ad}}
\newcommand{\rank}{{\rm rank}}
\newcommand{\0}{\mathbf{0}}
\newcommand{\m}{\mathbf{n}}
\newcommand{\IM}{{\rm Im}\,}
\newcommand{\ie}{{\em i.e.,} }
\newcommand{\eg}{{\em e.g.,} }
\newcommand{\ddx}{\frac{\partial}{\partial x}}
\newcommand{\ddy}{\frac{\partial}{\partial y}}
\newtheorem{thm}{Theorem}[section]
\newtheorem{cor}[thm]{Corollary}
\newtheorem{lem}[thm]{Lemma}
\newtheorem{prop}[thm]{Proposition}
\newtheorem{exm}[thm]{Example}
\theoremstyle{definition}
\newtheorem{defn}[thm]{Definition}
\newtheorem{rem}[thm]{Remark}
\newtheorem{notation}[thm]{Notation}
\numberwithin{equation}{section}
\def\be {\begin{equation}}
\def\ee {\end{equation}}
\def\ba {\begin{eqnarray}}
\def\ea {\end{eqnarray}}
\def\bes {\begin{equation*}}
\def\ees {\end{equation*}}
\def\bas {\begin{eqnarray*}}
\def\eas {\end{eqnarray*}}
\def\bpr {\begin{proof}}
\def\epr {\end{proof}}

\begin{document}
\baselineskip=18pt
\renewcommand {\thefootnote}{\dag}
\renewcommand {\thefootnote}{\ddag}
\renewcommand {\thefootnote}{ }
\pagestyle{empty}
\begin{center}
                \leftline{}
                \vspace{-0.20 in}
{\Large \bf Parametric normal forms for Bogdanov--Takens singularity; the generalized saddle-node case
%\\ [0.5ex]
} \\ [0.3in]

{\large Majid Gazor$^{*}$ and Mojtaba Moazeni }
\footnote{$^*\,$Corresponding author. Phone: (+98-311) 3913634; Fax: (+98-311) 3912602; Email: mgazor@cc.iut.ac.ir}

\vspace{0.15in}
{\small {\em Department of Mathematical Sciences,
Isfahan University of Technology
\\[-0.5ex]
Isfahan 84156-83111, Iran
}}

\today
\vspace{-0.4in}

\end{center}
\vspace{0.4in}
\noindent
\rule{7.10in}{0.012in}
\vspace{0.1in}

\noindent  We obtain a parametric normal form for any non-degenerate perturbation of the generalized saddle-node case of Bogdanov--Takens singularity.
Explicit formulas are derived and greatly simplified for an efficient implementation in any computer algebra system. A Maple program is prepared for an automatic parametric normal form computation. A section is devoted to present some practical formulas which avoid technical details of the paper.

\vspace{0.10in}

\vspace{0.10in}

\section{Introduction}

Any planar differential system
\bes
[\dot{x}, \dot{y}]:= \Le[\sum^\infty_{i+j=2} a_{ij} x^iy^j, -x+\sum^\infty_{i+j=2} b_{ij}x^iy^j\Ri],
\ees is associated with the vector field
\be\label{eq1}
v(x, y):= -x\ddy+ {\rm h.o.t},\quad x, y\in \mathbb{R}
\ee where {\rm h.o.t} denotes nonlinear (higher order) terms in \((x, y)\), and vice versa. Given this, the words \emph{vector fields} and \emph{differential systems} are interchangeably used. A classical normal form of the (unperturbed) Bogdanov--Takens singularity \eqref{eq1} is
\be\label{eq2} v^{(1)}(x,y)= -x\ddy+\sum^\infty_{k=1} a_{k} y^{k+1}\ddx+ \sum^\infty_{k=1} b_{k}y^{k}\Le(x\frac{\partial}{\partial x}+ y\frac{\partial}{\partial y}\Ri).\ee Assume that there exist \(i\) and \(j\) such that \(a_i, b_j\neq0\) and define
\be\label{rs} r_1:=\min\{k\,|\,a_k\neq0\}\hbox{ and } s:=\min\{k\,|\,b_k\neq0\};\ee
see \cite{baidersanders} where the authors use \(\mu\) for \(r_1,\) and \(\nu\) for \(s\).
%Note that \(b_k\) for \(k>s\) will be eliminated in parametric hyper-normalization and that is why we only use \(s\) as opposed to \(r_1.\)
Throughout this paper we assume that \(2s<r_1\). Following Str\'{o}\.{z}yna and \.{Z}oladek we call this case \textit{the generalized saddle-node case} of Bogdanov--Takens singularity; see \cite{Stroyzyna,Zoladek02,Zoladek03} for more details. Most real life applications of Bogdanov-Takens singularity appear in spaces with more than two dimensions. Here, we have assumed that a center manifold reduction has already been applied.

Our main result in this paper is as follows. For any (\(N\)-degree truncated) multiple-parametric non-degenerate perturbation of the generalized saddle-node case of Bogdanov--Takens singularity \(v(x,y, \mu)\), there exist invertible parametric changes of state variables, time rescaling and reparametrization that transform \(v(x,y, \mu)\) into the truncated parametric normal form
\ba\nonumber
\dot{x}&= & \mu_{s}+\mu_{s+1}y+ \sum^n_{i=1} (\alpha_{j_i}+\mu_{i+s+1}) y^{j_i+1}+xy^s+\sum^{s-2}_{k=0} \mu_{k+1}xy^{k},
\\\label{PNF0}
\dot{y}&=& -x+ y^{s+1}+\sum^{s-2}_{k=0} \mu_{k+1}y^{k+1}.
\ea Here \(j_i\neq r_1+s^2+s\) and \(j_i\neq k(s+1)+2s\) for \(k\neq s, k\geq 0,\) when \(r_1=\min\{j_i\,|\, \alpha_{j_i}\neq0\}\) and \(r_1\neq s(s+1)+2s.\)

\pagestyle{myheadings}
\markright{{\footnotesize {\it  M. Gazor and M. Moazeni
\hspace{3.0in} {\it Bogdanov--Takens singularity }}}}

There are only a few research results dealing with parametric cases of different singularities; see \cite{GazorMokhtariInt,GazorMokhtariEul,GazorYuSpec,GazorYuFormal,Murd98,Murd09Mal,PeiYu02SingleZero,yl}. However, there is no research report in the literature dealing with parametric hypernormalization of the Bogdanov--Takens singularity in which parametric changes of state variable, parametric time rescaling and reparametrization are all efficiently used.

Baider and Sanders \cite{BaidSand91,baidersanders} considered the simplest normal form computation of Bogdanov--Takens singularity without parameters. Since then many theoretical and computational methods have been employed to solve the remaining unsolved cases or to develop computer programs for its practical applications; see \eg \cite{Algaba,baiderchurch,kw,pwang,Stroyzyna,wlhj,yy2001,yy2001b,Zoladek02,Zoladek03,chendora,ChenWang}. Baider and Sanders \cite{baidersanders} chose to smartly skip tedious derivation of most formulas. Here we have used results of \cite{baidersanders} while we have derived the involved formulas and extended them to include time rescaling cases. Formulas are substantially simplified and applied to develop our Maple program.

Recently, Str\'{o}\.{z}yna and \.{Z}oladek \cite{ZoladekNF2013} have claimed to conclusively solve the simplest normal forms of Bogdanov--Takens singularity without parameters. Str\'{o}\.{z}yna and \.{Z}oladek \cite{Stroyzyna,Zoladek02,Zoladek03} obtained complete formal orbital normal form classification for germs of {\em analytic complex} Bogdanov--Takens singularities through their powerful non-algebraic method. In addition, the convergence analysis of the classical normal forms are presented; see also \cite{StZolDiv}.

The rest of this paper is organized as follows. Section \ref{sec2} presents the necessary algebraic structures, notation and our theory of normal form computations. In Section \ref{sec3} we derive and simplify the necessary formulas for the computation of orbital normal forms. The simplest normal form obtained by Baider and Sanders are reproduced in Section \ref{sec4}. Under some technical conditions, Section \ref{sec5} deals with the parametric normal form of any non-degenerate perturbation of the system given by Equation (\ref{eq1}). Some practical formulas are included in Section \ref{sec6} that demonstrate the applicability of the results.

\section{Algebraic structures }\label{sec2}

The necessary algebraic structures are presented in this section. The same notation as of \cite{baidersanders} are used here, while they are extended to include time rescaling. Define
\ba\label{Alk}
A_{k}^{l}& :=&\frac{k-l+1}{k+2}x^{l+1}y^{k-l}\frac{\partial}{\partial
x}-\frac {l+1}{k+2}x^ly^{k-l+1}\frac{\partial}{\partial y} \qquad\qquad\; (-1\leq l\leq k+1),
\\\label{Blk}
B_{k}^{l}&:=& x^{l+1}y^{k-l}\frac{\partial}{\partial x}+x^ly^{k-l+1}\frac{\partial}{\partial y} \qquad\qquad\qquad\qquad\qquad\qquad\quad\; (0\leq l\leq k),
\\
Z^l_k&:=&x^ly^{k-l} \qquad\qquad\qquad\qquad\qquad\qquad\qquad\qquad\qquad\qquad\quad\;\; (0\leq l\leq k),
\ea and denote the ring of all formal power series in terms of \(x\) and \(y\) by \(\R:=F[[x, y]],\) where \(F\) denotes the field of real numbers \(\mathbb{R}\) in Sections \ref{sec3} and \ref{sec4}. Here, \(\R\) stands for the space of all time rescaling generators. We refer to either of \(A_{k}^{l}, B_{k}^{l}\) and \(Z_{k}^{l}\) as a \emph{term} in this paper. Hence,
\be\label{R}
\R:= \Big\{\sum c_{l,k}Z^l_k\,\big|\, {0\leq l\leq k}, c_{l,k}\in F\Big\}.\ee Further,
\be\label{L} \LST:= \Big\{a^1_0A^{1}_0+ \sum a^l_kA^l_k+\sum b^n_mB^n_m| a^l_k, b^n_m\in F\Big\},\ee where the first summation is over \(-1\leq l\leq k+1\) with \(k\geq 0,\) and the indices of the second summation are subjected to the conditions \(0\leq n\leq m\) and \(m\geq0.\) The space \(\LST\) represents the space of all vector fields of type Equation \eqref{eq1}.

\begin{rem} Equations \eqref{Alk} and \eqref{Blk} were derived via a \(\Sl\)-representation for planar vector fields where the triad \(M, N, H\) are given by \(M:= y\ddx,\) \(N:=-x\ddy, \) and \(H:= [M, N]\); see \cite{BaidSand91} for details. The vector field \(B^n_m\) is called an {\it Eulerian vector field} and does not have any nontrivial first integral. For any \(l\) and \(k,\) the vector field \(A_{k}^{l}\) is called a {\it Hamiltonian vector field} and its Hamiltonian is given by \(h^l_k:= \frac{1}{k+2}x^{l+1}y^{k-l+1}.\) In particular, the algebra of first integrals for \(A^{-1}_k\) is generated by \(y.\)
\end{rem}

\begin{lem}[Structure Constants]\label{StrCons} Following \cite{baidersanders}, the Lie algebraic structure of \(\LST\) is governed by
\bas {[A^l_k, A^n_m]}&=&(k+m+2)\Le(\frac{n+1}{m+2}-\frac{l+1}{k+2}\Ri)A^{l+n}_{k+m},
\\
{[A^l_k, B^n_m]}&=&\frac{m(m+2)}{(m+k+2)}\Le(\frac{n}{m}-\frac{l+1}{k+2}\Ri)
B^{n+l}_{k+m}-kA^{n+l}_{k+m},
\\
{[B^l_k, B^n_m]}&=&(m-k)B^{n+l}_{k+m}.
\eas
The space of all near-identity time rescaling transformations acts on the Lie algebra through a left \(\R\)-module structure (see \cite{GazorYu}) given by the following rules
\bas
Z^n_mA^l_k&=& A^{n+l}_{m+k}+ \frac{(k+2)n-m(l+1)}{(k+2)(k+m+2)}B^{n+l}_{m+k},
\\
Z^n_mB^l_k&=& B^{n+l}_{m+k},
\\
Z_m^nZ^l_k&=&Z^{n+l}_{m+k}.
\eas
\end{lem}

The classical normal forms of Bogdanov--Takens singularity in their
Hamiltonian--Eulerian decomposition are presented in the following lemma; see \cite{baidersanders}.
\begin{lem}\label{cl}
For any vector field \(v\) given by (\ref{eq1}), there exist invertible changes of state variables transforming \(v\) into the first level normal form
\be\label{clsAB}
v^{(1)}= A^1_0+\sum^\infty_{k=r_1} a_kA^{-1}_k+ \sum^\infty_{k=s} b_kB^{0}_k.
\ee In particular, nonlinear part of \(v^{(1)}\) commutes with \(M=A^{-1}_0,\) \ie \(v^{(1)}-A^1_0\in \ker \ad_M.\)
\end{lem}

The \(\Sl\)-style normal form is used in Lemma \ref{cl} and we extend it to a simplest normal form style via notion of formal basis style, that is, we give priority to elimination of \(B^0_k\)-terms over \(A^{-1}_k\)-terms when such choices are needed; see \cite{GazorYuSpec,GazorYuFormal,Murd09Mal}. This style provides a unique complement space (or a unique projection) for \({\rm im }\, d^{n, N} \) within the space of \(A^{-1}_p\) and \(B^0_q\) terms.

\begin{defn} We extend the grading function introduced in \cite{baidersanders} to include parameters by
\be
\delta\Le(A^l_k\mu^\m\Ri)=\delta\Le(B^l_k\mu^\m\Ri):= k+ls+\Le(r_1+2\Ri)|\m|, \quad |\m|:= \sum^p_{i=1} m_i, \hbox{ for } \m:= (m_1, m_2, \ldots, m_p),
\ee  where \(\mu^\m:=\prod^p_{i=1}{\mu_i}^{m_i}\). This grading is generalized to include time rescaling terms by
\be
\delta\Le(Z^l_k\mu^\m\Ri):= k+ls+(r_1+2)|\m|.
\ee Hence, \(\LST=\sum\LST_i\) is a graded Lie algebra and \(\R=\sum\R_i\) is a graded ring. Furthermore, \(\LST\) is a left \(\R\)-graded module; see \cite{GazorYuSpec} for more details.
\end{defn}
Let \(\alpha:= {\rm sign}(b_s)\) and \(\beta:=\frac{1}{\sqrt[s+1]{b_s\alpha^s}}.\) Thus, by the time rescaling \(t:=\beta \tau\) and state changes of variables \((x,y):=(\alpha X,\beta Y),\) Equation (\ref{clsAB}) is transformed into
\bes
\tilde{v}^{(1)}= A^1_0+ B^0_s+\sum^\infty_{k=s} \frac{\beta^{k+2} a_k}{{\rm sign}(b_s)}A^{-1}_k+ \sum^\infty_{k=r_1} \beta^{k+1} b_kB^{0}_k.
\ees Hence, without loss of generality we assume that \(b_s=1\) in Equation (\ref{clsAB}). Denote \be\mathbb{B}_s:=A^1_0+ B_s^0\in \LST_s.\ee
\begin{notation}[Pochhammer \(k\)-symbol]
The Pochhammer \(k\)-symbol notation is defined by
\bes(a)^n_{k}= a (a+k)(a+2k)\cdots \big(a+(n-1)k\big),\ees for any natural number \(n\) and real number \(k.\)
\end{notation}
Now we present our theory of hypernormalization which we shall use in the next three sections. Let \(v= \sum^\infty_{i=s} v_i, v_i\in \LST_i,\) and let \(A:= \sum^\infty_{i=0} A_i\) denote a grading for the space of all transformation generators. The space \(A\) shall be defined in each of the following three sections. The space \(A\) acts on \(\LST\) and we denote this action by \(*.\) Define
\bas
d^{n,s+1}: A_{n-s} &\rightarrow& \LST_n,
\\
d^{n,s+1}(Y_{n-s})&:=& Y_{n-s}*v_s.
\eas Here, \(v_s:= \mathbb{B}_s.\) Assume that a normal form style is fixed; a {\it normal form style} refers to a rule on how to uniquely assign a complement space to any given vector subspace of \(\LST\). Then, there exist spaces of \(\RC^{n, s+1}:= \IM d^{n,s+1}\) and \(\C^{n, s+1}\) such that \(\LST_n:= \RC^{n, s+1}\oplus \C^{n, s+1}.\) Therefore, \(v\) can be transformed into the \((s+1)\)-th level (orbital or parametric; depending on the transformation space \(A\)) normal form \(v^{(s+1)}= \sum^\infty_{k=s}v_k,\) where \(v_n\in \C^{n, s+1}\) for all \(n>s.\) Then, for any \(n\) and \(N\) such that \(n\geq N\geq s+1,\) we inductively define
\ba\nonumber
d^{n,N}: \ker d^{{n-s}-1, N-1}\times A_{n-s}&\rightarrow& \LST_{n},
\\\label{LnN}
d^{n,N}\big(g_{n-N+1}, \ldots, g_{n-s}\big)&:=&\sum^{N-1}_{k=s} g_{n-k}*v_{k}.
\ea
Here, \(d^{n, N}\) projects \bes \big(g_{n-N+1}, \ldots, g_{n-s}\big)*\Le(\sum^{N-1}_{j=s} v_j\Ri):=\Le(\sum^{n-s}_{i=n-N+1} g_i\Ri)*\Le(\sum^{N-1}_{j=s} v_j\Ri) \ees into \(\LST_n.\) For any \(n\) with \(s+1\leq n< N,\) let \(d^{n,N}:= d^{n,N-1}.\) Denote
\bes
\RC^{n, N}:= \IM d^{n,N} \quad \hbox{ and }\quad \LST_n:= \RC^{n, N}\oplus \C^{n, N},
\ees where \(\C^{n, N}\) is the unique complement space of \(\RC^{n, N}\) in \(\LST_n.\) A vector field \(v=\sum v_k\) is called an {\it \(N\)-th level normal form} when \(v_k\in \C^{k,N}\) for any \(k\in \N,\) or an {\it infinite level normal form} when \(v_k\in \C^{k,k}\) for any \(k\in \N.\) Then, by \cite[Theorems 4.3--4.4]{GazorYuSpec} the following theorem holds.
\begin{thm} \label{NthONF}
For any \(v\in \LST,\) there exist near-identity changes of variables such that they transform \(v\) into its \(N\)-th level \(v^{(N)}\) or infinite level \(v^{(\infty)}\) (parametric or orbital) normal form.
\end{thm}

%\begin{rem} Assume that we only consider changes of state variables (with no time rescaling or reparametrization) and that \((N-1)\)th level normal form is
%already computed. The space \(\ker d^{{n-s}-1, N-1}\) consists of the transformation generators available for hypernormalization in the \(N\)th stage; see
%\cite{GazorYuSpec}. For any \((g_{n-N+1}, \ldots, g_{n-s-1})\) from this space, \(\sum g_i\) commutes with \(\sum^{N-2}_{i=s}v_i\) modulo grade \(n\). Yet by
%applying these transformations upon the \((N-1)\)th level normalized vector field, the vector field does not remain in the classical normal form. In order to
%bypass this phenomenon, for any \(g\in \ker d^{{n-s}-1, N-1}\times A_{n-s}\) we can always find a \(Y_{n-s}\in \LST_{n-s}\) such that it transforms the vector
%field back to the classical normal form, \ie
%\bes j^n\Le(\Le[\sum^{N-1}_{j=s}v_j, \sum_{i=n-N+1}^{n-s} g_i\Ri]+ [v_s, Y_{n-s}]\Ri)\in \ker \ad_M,\ees where \(j^nv\) denotes the \(n\)-jet of \(v.\) This is
%the basic idea of many formulas in this paper.
%\end{rem}

The following lemma highlights the basic idea used in many formulas in this paper.

\begin{lem}\label{23}
Let \(v^{(N-1)}=\sum^\infty_{i=s}v_i\) (\(v_i\in \LST_i\)) be the \((N-1)\)-th level (parametric or orbital) normal form of \(v.\) Suppose that for any \(g\in \ker d^{n-1, N-1}\times A_{n-s},\) there exists a \(Y_{n-s}\in A_{n-s}\) such that
\bes d^{n, N}\left(g*v^{(N-1)}+[v_s, Y_{n-s}]\right)\in \Span\Le\{A^{-1}_p, B^0_q\,|\, p, q\in \mathbb{N}\Ri\}\cap \LST_n.\ees
Then,
\bes
\RC^{n, N}= \RC^{n, N-1}+ \Span \Le\{d^{n, N}\left(g*v^{(N-1)}+[v_s, Y_{n-s}]\right)\,\big|\, g\in \ker d^{n-1, N-1}\times A_{n-s}\Ri\}.
\ees
\end{lem}
The space \(\RC^{n, N}\) is sometimes called the {\it \(N\)th level removable space} of grade \(n.\)

The following two lemmas play a central role in development of our Maple program; compare them with  \cite[Propositions 6.1.2 and 6.1.5]{BaidSand91}.

\begin{lem}\label{AReduce} For any nonnegative integers \(m\) and \(n,\) there exists a \(\delta\)-homogenous polynomial vector field \(\mathfrak{A}^n_m\in\LST\) such that
\bas
A^n_m+ {[\mathbb{B}_s, \mathfrak{A}^n_m]}&=&\Le(\frac{(s+2)(m+1+ns)(m)^{n-1}_{s}}{(m+2-n)^{n+1}_{s+1}}
- \frac{(s+2)(m)^{n-1}_{s}}{(m+2)(m-n+1)^{n-1}_{s+1}}
\Ri) B_{m+ns}^0
\\&&+ \frac{(m)^{n+1}_{s}}{(m-n+2)^{n+1}_{s+1}} A_{m+ns+s}^{-1}.
\eas
\end{lem}
\bpr
By Lemma \ref{StrCons} we have
\ba\nonumber
\mathfrak{A}^n_m&:=& \sum_{q=1}^{n-1}\frac{s(s+2)(n+1-q)(m)^{q-1}_{s}}{(m+(q-1)s+2)^{2}_{s}(m-n+2)^{q}_{s+1}}
\sum_{l=0}^{n-1-q} \frac{\big(m+(q-1)s\big)^{l}_{s}}{\big(m-n+1+q(s+1)\big)^{l+1}_{s+1}}
B_{m+qs+ls}^{n-1-q-l}
\\\label{Atrans}&&+\sum_{l=0}^{n} \frac{(m)^{l}_{s}}{(m-n+2)^{l+1}_{s+1}}A_{m+ls}^{n-1-l}
\ea while
\bas
&&\sum_{q=1}^{n-1}\frac{s(s+2)(n+1-q)\big(m-n+2+s\big)^{q-1}_{s+1}}{
(m+(q-1)s+2)^{2}_{s}(m-n+2)^{q}_{s+1}}
\sum_{l=0}^{n-1-q} \frac{(m)^{l+q-1}_{s}}{\big(m-n+2+s\big)^{l+q}_{s+1}}
B_{m+qs+ls}^{n-1-q-l}
\\&=&\sum_{q=1}^{n-1}\frac{s(s+2)(n+1-q)\big(m-n+2+s\big)^{q-1}_{s+1}}{
(m+(q-1)s+2)^{2}_{s}(m-n+2)^{q}_{s+1}}
\sum_{l=q}^{n-1} \frac{(m)^{l-1}_{s}}{\big(m-n+2+s\big)^{l}_{s+1}}
B_{m+ls}^{n-1-l}
\\&=&\sum_{l=1}^{n-1} \frac{(m)^{l-1}_{s}}{\big(m-n+2+s\big)^{l}_{s+1}}
\sum_{q=1}^{l}\frac{s(s+2)(n+1-q)\big(m-n+2+s\big)^{q-1}_{s+1}}{
(m+(q-1)s+2)^{2}_{s}(m-n+2)^{q}_{s+1}}
B_{m+ls}^{n-1-l}
\\&=&
\sum _{l=1}^{n-1}\Le(\frac{(s+2)(m)^{l-1}_{s}}{(m+ls+2)(m-n+2)^{l}_{s+1}}
-\frac{(s+2)(m)^{l-1}_{s}}{(m+2)(m-n+2+s)^{l}_{s+1}}\Ri) B_{m+sl}^{n-1-l}.
\end{eqnarray*}
This completes the proof.
\end{proof}
\begin{lem} \label{BReduce} For natural numbers \(n\) and \(m,\) there exists a \(\mathfrak{B}^n_m\in \LST_{m+ns-s}\) such that
\bas
B^n_m+ {[\mathbb{B}_s, \mathfrak{B}^n_m]}&=& \frac{(m-s)^{n}_{s}}{(m-n+1)^{n}_{s+1}} B_{m+ns}^0.
\eas
\end{lem}
\begin{proof} Define
\be\label{Btrans}
\mathfrak{B}^n_m:=\sum _{l=0}^{n-1} \frac{(m-s)^{l}_{s}}{(m-n+1)^{l+1}_{s+1}}
B_{m+ls}^{n-1-l},
\ee and observe that
\(
{[\mathbb{B}_s, \mathfrak{B}^n_m]}= \sum_{l=0}^{n-1}\left(\frac{(m-s)^{l+1}_{s}}{(m-n+1)^{l+1}_{s+1}}-
\frac{(m-s)^{l}_{s}}{(m-n+1)^{l}_{s+1}}\right)B_{m+ls}^{n-l}.
\)
\end{proof}

The above two lemmas imply that transformation generators \eqref{Atrans} and \eqref{Btrans} simplify all nonlinear terms except for \(A^{-1}_p\) and \(B^0_q\).

\section{The orbital normal forms }\label{sec3}

In this section we compute the orbital normal form of the generalized saddle-node case of the vector field (\ref{eq1}). Let \(A:= \LST\oplus\R.\) Hence, Equation (\ref{LnN}) is given by
\bas
d^{n, N}\big(S_{n-N+1}, \ldots, S_{n-s}, T_{n-N+1}, \ldots, T_{n-s}\big)&:=&\sum^{N-1}_{k=s} \big([v_{k},S_{n-k}]+ T_kv_{n-k}\big), \;\;\; S_i\in \LST_i, \; T_i\in \R_i,
\eas where the order of state terms and time terms are rearranged in this section for convenience.

We generalize a map \(\Gamma\) defined in \cite{baidersanders} to include the space \(\R\) in its domain, \ie
\bas
\Gamma: \LST\times \R \rightarrow \LST,
\eas is given by
\bas
\Gamma(S, Z^n_m)&:=&\Le[[\mathbb{B}_s, S]+ T\,\mathbb{B}_s, A^{-1}_0\Ri], \quad \hbox{ for any } T\in \R \hbox{  and }S\in \LST.
\eas The map \(\Gamma\) is a bilinear map. Baider and Sanders \cite{baidersanders} indicated that \(A^{k+1}_k\) and \(B^k_k\) (for any \(k\in \N\)) generate vector fields \(\mathcal{A}^{k+1}_k\) and \(\mathcal{B}^{k}_k\) such that \(\ker(\Gamma)\) is spanned by \(\mathcal{A}^{k+1}_k\) and \(\mathcal{B}^{k}_k.\) Further, \(Z^l_k\) (for \(l\leq k\)) also generate vector fields \(\mathcal{A}^{l}_k\) such that
\ba\label{kerG}\ker(\Gamma)=\Span \Big\{\big(\mathcal{A}^l_k, \Le(k-l+1\Ri)Z^l_k\big), (\mathcal{B}^{k}_k, 0)\,|\, l\leq k+1, l, k\in \N\cup\{0\}\Big\}.\ea
Here, our formulas uniformly treats the vector fields \(\mathcal{A}^l_k\) for \(l\leq k\) and \(l=k+1.\)
We derive the formulas for \(\mathcal{A}^l_k\) and \(\mathcal{B}^k_k\) in Equations (\ref{Acal}) and (\ref{Bcal}). Only these vector fields may contribute to further normalization of the classical normal forms; also see \cite[Proposition 5.1.1]{BaidSand91} and \cite{Moazeni}.

\begin{lem}\label{Znm}
For any nonnegative integers \(m, n,\) there exists a \(\delta\)-homogenous \(\mathcal{A}^n_m\) such that
\be
\left(\mathcal{A}^n_m, (m-n+1) Z^n_m\right)\in\ker(\Gamma)
\ee where
\begin{eqnarray}\nonumber
\mathcal{A}^n_m&=&
\sum_{l=0}^{n-1}
\frac{(s+2)(m)^{l}_{s}}{(m+(l+1)s+2)(m+2-n+s)^{l}_{s+1}}
B_{m+ls+s}^{n-l-1}-\frac{(m-n+1)}{m+2} B^{n}_{m}
\\\label{Acal}&&
+\sum_{l=-1}^{n} \frac{(m)^{l+1}_{s}}{(m+2-n+s)^{l+1}_{s+1}} A_{m+s+ls}^{n-l-1}.
\end{eqnarray}
\end{lem}
\begin{proof} By definition of \(\Gamma\) and the structure constants we have
\begin{eqnarray}\nonumber
\mathcal{A}^n_m&:=&\sum_{l=-1}^{n}
\frac{(m)^{l+1}_{s}}{(m+2-n+s)^{l+1}_{s+1}} A_{m+s+ls}^{n-l-1}
-\sum_{l=0}^{n}\frac{(m-n+1)(m-s)^{l}_{s}}{(m+2)(m-n+s+1)^{l}_{s+1}}
B_{m+ls}^{n-l}
\\\nonumber&&
+\sum_{q=0}^{n-1} \sum_{l=0}^{n-1-q}\frac{s(s+2)(n+1-q)(m)^{q}_{s}\big(m+qs\big)^{l}_{s}}{(m+qs+2)^{2}_{s}
(m+2-n+s)^{q}_{s+1}\big(m-n+(q+1)(s+1)\big)^{l+1}_{s+1}} B_{m+s+qs+ls}^{n-1-q-l}
\\\label{Acalnm}&&+\sum_{l=0}^{n-1}\frac{(m-n+1)(m)^{l}_{s}}{(m-n+s+1)^{l+1}_{s+1}}B_{m+s+ls}^{n-l-1}.
\end{eqnarray} We conclude the proof by simplifying \eqref{Acalnm}.
\end{proof}

\begin{lem}\label{ZnmXs} For any \(m, n\) \((0\leq n\leq m+1)\) we have
\be\label{Righthand}
d^{m+ns+s,s+1}\big(\mathcal{A}^{n}_m, (m-n+1)Z^n_m\big)= a^{s}_{m,n}A_{m+2s+ns}^{-1}+ a^{s}_{m,n}c^{0, s}_{m,n+2}B^{0}_{m+ns+s},
\ee where
\be \label{arsmn}
a^{s}_{m,n}:=\frac{(m)^{n+2}_{s}}{(m+2-n+s)^{n+1}_{s+1}} \hbox { and } c^{0, s}_{m,n+2}:= \frac{(s+2)\big(m+(n+1)s+1\big)}{(m+ns)\big(m+(n+1)s\big)}.
\ee
%In particular, \(a^{s}_{m,n}\neq 0, c^{0, s}_{m,n+2}\neq 0\)  for any \(m, n\) and we have
%\bas
%{[\mathbb{B}_s, \mathcal{A}^{m+1}_m]}&=& a^{s}_{m,m+1}A_{m+(m+3)s}^{-1}
%+ a^{s}_{m,m+1}c^{0, s}_{m,m+3}B^{0}_{m+(m+2)s}.
%\eas
\end{lem}
\begin{proof} Proof follows from Lemma \ref{Znm} and
\bes
[\mathbb{B}_s, \mathcal{A}^n_m]+ (m-n+1)Z^n_m\mathbb{B}_s=
a^{s}_{m,n}A_{m+2s+ns}^{-1}
+ a^{s}_{m,n}c^{0, s}_{m,n+2}B^{0}_{m+ns+s}.
\ees
\end{proof}

\begin{cor}\label{Z0} \(\C^{m+s,s+1}\subseteq \Span\{A^{-1}_{m+2s}\}\) for any \(m\in \N,\) where \(B\)-terms precede \(A\)-terms in our normal form style.
\end{cor}
\bpr
The proof follows from Lemma \ref{ZnmXs}; using \(\big(\mathcal{A}^{0}_m, (m+1)Z^0_m\big),\) all \(B^{0}_{m+s}\) terms can be eliminated in the \((s+1)\)-th level orbital normal form.
\epr

\begin{rem} \label{updatedr1}
Corollary \ref{Z0} implies that \(\big(\mathcal{A}^{0}_m, (m+1)Z^0_m\big)\) can be used to eliminate all \(B^0_k\)-terms for \(k>s\). This may accidentally simplify the coefficient of \(A^{-1}_{r_1}\) or changes a zero coefficient (of \(A^{-1}_r\) for a \(r<r_1\)) to a nonzero coefficient.
Thus, we need to iteratively update the number \(r_1\). For a given number \(r_*\) assume that \(A^{-1}_{r}\) and \(B^0_{r-s}\) (for \(2s<r<r_*\)) are eliminated from the system. Denote \(a^{-1}_{2s+m}\) and \(b^{0}_{s+m}\) for the updated coefficients of \(A^{-1}_{2s+m}\) and \(B^0_{s+m}.\) Then,
\be\label{r1}
r_1:=\min \Le\{2s+m\;\big|\; a^{-1}_{2s+m}\neq \frac{m(m+s)b^{0}_{s+m}}{(s+2)(s+m+1)}, m\in \N, m\geq r_*-2s\Ri\}.
\ee
Hence, the updated \(r_1\) generically is \(2s+1\); see \cite[Page 2160]{ChenWang} and \cite{GazorMokhtari} for similar remarks. This update is not further necessary once all
\(B^{0}_{r-s}\)-terms (for \(r\leq r_1\)) are eliminated.
%The formula \eqref{r1} is centrally important for implementation of the results in a computer program.
\end{rem}

\begin{cor}\label{Bk} For any \(k\geq 1\) and \(N\geq s+1\) we have \(\C^{k(s+1)+s,N}=\{0\}\) if and only if \(k\neq s.\) Furthermore, \(\C^{s^2+2s,N}=\Span \{A_{s^2+3s}^{-1}\}\) for any \(N\geq s+1.\)
\end{cor}
\begin{proof} The term \(B^{k}_k\) generates a polynomial vector field \(\mathcal{B}^{k}_k\) such that \((\mathcal{B}^{k}_k,0)\in \ker(\Gamma).\) Indeed,
\be\label{Bcal}
\mathcal{B}^{k}_k:= \sum _{l=0}^{k} \frac{(k-s)^{l}_{s}}{l!(s+1)^l}B_{k+ls}^{k-l},
\ee and by
\(
[\mathbb{B}_s,\mathcal{B}^{k}_k]= \sum _{l=0}^{k} \frac{(k-s)^{l+1}_{s}}{l!(s+1)^l}B_{k+ls+s}^{k-l}-
l\frac{(k-s)^{l}_{s}}{l!(s+1)^{l-1}}B_{k+ls}^{k-l+1},
\) we have
\ba\label{XsBkk}
{[\mathbb{B}_s,\mathcal{B}^{k}_k]}&=&\frac{(k-s)^{k+1}_{s}}{k!(s+1)^k}B_{{k(s+1)+s}}^0.
\ea Thus for any \(k>0,\) Lemma \ref{ZnmXs} and Equation (\ref{XsBkk}) imply that \(A_{k(s+1)+2s}^{-1}, B_{k(s+1)+s}^0\in \IM d^{k(s+1)+s}_{s+1}\) if and only if \(k\neq s.\)
This completes the proof.
\end{proof}
%Now we are ready to present the \((s+1)\)th level orbital normal form.
\begin{prop}\label{s+1level}
The generalized saddle-node case of Bogdanov--Takens singularity (\ref{eq1})
can be transformed into the \((s+1)\)-th level orbital normal form
\ba A^1_0+\beta_s B^0_s+ \alpha_{r_1} A^{-1}_{r_1}+ \sum \alpha_i A^{-1}_i,\ea
where the summation is over \(i>r_1\), and \(\alpha_{k(s+1)+2s}=0\) for any natural number \(k\neq s.\) Besides, \(r_1\neq k(s+1)+2s\) for \(k\neq s.\)
\end{prop}
\bpr The proof follows from Lemma \ref{ZnmXs} and Corollary \ref{Bk}. \epr
Following Lemma \ref{ZnmXs}, for two different pair of values \(m_1, n_1\) and \(m_2, n_2\), (\(m_1+n_1s=m_2+n_2s\)) we end up with the same pair \(A_{m_1+2s+n_1s}^{-1}\) and \(B^{0}_{m_1+n_1s+s}\) in the right hand side of Equation \eqref{Righthand}. Since \(c^{0, s}_{m_1,n_1+2}= c^{0, s}_{m_2,n_2+2}\), the associated homogeneous transformation terms can not simplify \(A_{m_1+2s+n_1s}^{-1}\) and therefore, they generate an element in the symmetry group of \(\mathbb{B}_s\) as follows. For nonnegative integers \(m_j, n_j\) (\(j=1, 2,\) \(n_j\leq m_j, m_1+n_1s=m_2+n_2s\)) denote
\ba\label{A2}
\mathcal{A}^{n_{1, 2}}_{m_{1, 2}}&:=& \frac{1}{a^{s}_{{m_1},{n_1}}}\mathcal{A}^{n_1}_{m_1}- \frac{1}{a^{s}_{{m_2},{n_2}}}\mathcal{A}^{n_2}_{m_2},
\\\label{Z2}
Z^{n_{1, 2}}_{m_{1, 2}}&:=& \frac{({m_1}-{n_1}+1)}{a^{s}_{{m_1},{n_1}}}Z^{n_1}_{m_1}
-\frac{({m_2}-{n_2}+1)}{a^{s}_{{m_2},{n_2}}}Z^{n_2}_{m_2}.\ea
Hence, \(\delta(\mathcal{A}^{n_{1, 2}}_{m_{1, 2}})=\delta(Z^{n_{1, 2}}_{m_{1, 2}})= m_1+n_1s\).

\begin{cor}\label{kerAk+1Z0} Let \(m_1+n_1s\neq s^2+s.\) Then,
\bes\ker d^{m_1+n_1s+s,s+1}=\Span\Le\{\Le(\mathcal{A}^{n_{1, 2}}_{m_{1, 2}}, Z^{n_{1, 2}}_{m_{1, 2}}\Ri)\big|\, m_1+n_1s=m_2+n_2s\Ri\},\ees
while for \(m_1+n_1s=s^2+s,\)
\bes \ker d^{s^2+2s,s+1}=\Span\Le\{\Le(\mathcal{A}^{n_{1, 2}}_{m_{1, 2}}, Z^{n_{1, 2}}_{m_{1, 2}}\Ri),(\mathcal{B}^s_s,0)\,\big|\, m_1+n_1s=m_2+n_2s\Ri\}.\ees
\end{cor}
\bpr
%Since \(m_1+n_1s=m_2+n_2s,\) we have \(c^{0, s}_{{m_1},{n_1}+2}=c^{0, s}_{{m_2},{n_2}+2}.\)
The proof is complete by Lemmas \ref{ZnmXs} and \ref{Bk}.
\epr
Most formulas are presented for an arbitrary natural number \(r\) whenever it is possible.

\begin{lem} \label{rss} For any \(N\geq r_1-s+1,\) we have \(\C^{r_1+s^2,N}=\{0\}.\)
\end{lem}
\bpr For any \(r,\) \([A_{r}^{-1},\mathcal{B}_{s}^{s}]=-r A_{s+r}^{s-1}+{\frac {s ( s+2)}
{s+r+2}}B_{s+r}^{s-1}\). By Lemmas \ref{AReduce}--\ref{BReduce} yield that there exists a \(Y^r_s\) such that
\ba\label{A-1rBcal}
[A_{r}^{-1},\mathcal{B}_{s}^{s}]+ [\mathbb{B}_s, Y^r_s]\!\!
&\!=\!&\!\!
\frac{(r+s)^s_{s}}{(r+3)^s_{s+1}} A_{s+r+{s}^{2}}^{-1}+f_{r,s} B_{r+s^2}^{0},
\ea
where
\be
f_{r,s}:=\frac{(r)^{s-1}_{s}}{(r+2)^{s-1}_{s+1}}+\frac{(s+2)(r+1+s^{2})(r+s)^{s-2}_{s}}{(r+3)^s_{s+1}}-{\frac {(s+2)(r+s)^{s-2}_{s}}{(r+s+2)(r+3+s)^{s-2}_{s+1}}}. \ee
This concludes the proof.
\epr
\begin{rem}\label{Rem3.9}
The above lemma implies that \(\mathcal{B}_{s}^{s}\) (of course, along with \(\mathcal{A}^{0}_{r_1+s^2-s}\) and \(Z^0_{r_1+s^2-s}\)) may simplify \(A_{s+r_1+{s}^{2}}^{-1}\) and \(B_{r_1+{s}^{2}}^{0}.\) Similarly by Corollary \ref{Bk}, \(\mathcal{B}_{k}^{k}\) eliminates both \(B_{{k(s+1)+s}}^0\) and \(A_{{k(s+1)+2s}}^{-1}\). Therefore, for \(r_1\neq 2s \mod{(s+1)}\), these do not contribute to \(\ker d^{N,r_1-s+1}\). The same is true for \(r_1\neq s(s+1)+2s\) by Proposition \ref{s+1level}. Hence, the only remaining case is \(r_1:= s(s+1)+2s\). This gives rise to
\be\label{3.13}
\left(\frac{({r_1}+3)^s_{s+1}}{({r_1}+s)^s_{s}}\mathcal{B}_{s}^{s}, \mathcal{B}_{s, {r_1}}, \frac{({r_1}+s^2-s+1)}{a^{{r_1},s}_{{r_1}+s^2-s,0}}Z^0_{{r_1}+s^2-s}\Ri)\in \ker d^{{r_1}+s^2-s, r_1-s+1},
\ee
where a \(\delta\)-homogeneous vector field \(\mathcal{B}_{s, {r_1}}\) is defined by
\bes  \frac{({r_1}+3)^s_{s+1}}{({r_1}+s)^s_{s}}Y^{r_1}_s+\frac{(2s)!(s+1)^{2s}}{(s)^{{2s}+1}_{s}}\Le(1-c^{0, s}_{{r_1}+s^2-s,2}-\frac{({r_1}+3)^s_{s+1}}{({r_1}+s)^s_{s}}f_{{r_1},s}\Ri)\mathcal{B}_{{2s}}^{{2s}}+\frac{1}{a^{{r_1},s}_{{r_1}+s^2-s,0}}
\mathcal{A}^{0}_{{r_1}+s^2-s}.\ees
%Note that \(\mathcal{B}_{s, {r_1}}\) is different from \(\mathcal{B}^{s_1}_{s_1, s_2}\) defined in Lemma \ref{Bs2}.
The left hand side of Equation \eqref{3.13} can be used to simplify
\(A_{s+r_2+{s}^{2}}^{-1}\) and \(B_{r_2+s^2}^{0}\) when \(r_2< 2{r_1}-2s= 2s(s+1)+2s\); see Proposition \ref{6.2} and Equation \eqref{PNF6D}.
\end{rem}

Now we deal with \(\Le(\mathcal{A}^{n_{1, 2}}_{m_{1, 2}}, Z^{n_{1, 2}}_{m_{1, 2}}\Ri)\) in the \((r_1-s+1)\)-th step of (orbital) hypernormalization.

\begin{lem}\label{XnmA-1} For arbitrary nonnegative integers \(r, m, n, n\leq m+1,\) we have
\bas
&&[\mathcal{A}^n_m, A^{-1}_r]+ (m-n+1) Z^n_mA^{-1}_r
\\&=&
\sum_{l=1}^{n}\Le(
\frac{(n-l+1)(r+m+2+ls)(m)^{l}_{s}}{(m+2+ls)(m+2-n+s)^{l}_{s+1}}
- \frac{r(s+2)(m)^{l-1}_{s}}{(m+2+ls)(m-n+2+s)^{l-1}_{s+1}}\Ri)
A_{r+m+ls}^{n-l-1}
\\&&
+(r+m+2)A_{r+m}^{n-1}
+\sum_{l=1}^{n}
\frac{(s+2)(n-l)(m)^{l-1}_{s}}{(r+m+2+ls)(m+2-n+s)^{l-1}_{s+1}} B_{r+m+ls}^{n-l-1}.
\eas
\end{lem}
\bpr
The proof follows from Lemma \ref{Znm}.
\epr

\begin{lem}\label{A-1Acal} There exists a \(\delta\)-homogenous polynomial vector field \(\mathcal{Y}^{n,r}_m\in \LST\) such that
\ba\label{A-1AcalEquation}
[A^{-1}_r, \mathcal{A}^n_m]+ (m-n+1)Z^n_mA^{-1}_r+ [\mathbb{B}_s,\mathcal{Y}^{n,r}_m]&=& \zeta^{r,s,n}_{m,n}A_{r+m+ns}^{-1}+ b^{r,s}_{m,n}c^{r,s}_{m,n} B_{r+m+(n-1)s}^0,\quad
\ea where \(r, m, n \in \mathbb{N}\cup\{0\}, \delta(\mathcal{Y}^{n,r}_m)= r+m+(n-2)s,\)
\ba \label{drsmn}
b^{r,s}_{m,n}:= \zeta^{r,s,n}_{m,n}- \frac{(r+m+ns+2)(m)^{n}_{s}}{(m+2-n+s)^{n}_{s+1}}, \; \hbox{ and }\;
c^{r,s}_{m,n}:= \frac{(s+2)\big(r+m+1+(n-1)s\big)}{\big(r+m+(n-2)s\big)^{2}_{s}},
\ea where \(\zeta^{r,s,n}_{m,n}\) follows Equation (\ref{zeta}) in Appendix.
\end{lem}
\bpr
Lemmas \ref{AReduce}, \ref{BReduce}, and \ref{XnmA-1} imply that there exists a \(\delta\)-homogenous \(\mathcal{Y}^{n,r}_m\in\LST\) such that Equation (\ref{A-1AcalEquation}) holds, where \(\frac{b^{r,s}_{m,n}c^{r,s}_{m,n}}{(s+2)}\) is given by Equation \eqref{brs} and \(b^{r,s}_{m,n}-\zeta^{r,s,n}_{m,n}\) equals to
\bas
&&\sum _{l=1}^{n-1} \frac{\big(r-m+s-ls+(l-n+s)r \big)(m)^{l-1}_{s}\big(r+m+ls\big)^{n-l}_{s}}{(m+2-n+s)^{l}_{s+1}
\big(r+m-n+2+l(s+1)\big)^{n-l}_{s+1}}-\frac{(r+m)^{n}_{s}}{(r+m-n+3+s)^{n-1}_{s+1}}
\\&&
- \frac{(r+m+2+ns)(m)^{n}_{s}}{\big(m+2+ns\big)(m+2-n+s)^{n}_{s+1}}+ \frac{r(s+2)(m)^{n-1}_{s}}{(m+ns+2)(m+2-n+s)^{n-1}_{s+1}}.
\eas
Now by telescoping series we have
\bas
&&
\sum _{l=1}^{n-1} \frac{\big(r-m+s-ls+(l-n+s)r \big)(m)^{l-1}_{s}\big(r+m+ls\big)^{n-l}_{s}}{(m+2-n+s)^{l}_{s+1}
\big(r+m-n+2+l(s+1)\big)^{n-l}_{s+1}}
\\&=&\frac{(r+m)^{n}_{s}}{(m-n+1)(r+m-n+3+s)^{n-1}_{s+1}}
-\frac{(r+m+ns-s)(m)^{n-1}_{s}}{(m-n+1)^{n}_{s+1}}.
\eas
This completes the proof.
\epr

\begin{cor}\label{ZnDirect} For nonnegative integers \(n_1, m_1, n_2, m_2\) such that \(n_1+m_1s=n_2+m_2s, n_1\neq n_2,\) let \(\mathcal{A}^{n_{1, 2}}_{m_{1, 2}}\) and \(Z^{n_{1, 2}}_{m_{1, 2}}\) follow Equations (\ref{A2}--\ref{Z2}), respectively. Then, for any nonnegative integers \(m_3\) and \(n_3\) such that \(m_3+n_3s={r_1}+m_1+n_1s-2s,\) there exists a state \(\delta\)-homogenous solution \(\mathcal{Y}^{n_{1, 2},{r_1}}_{m_{1, 2}}\) such that
\bas
\left[A^{-1}_{r_{1}}, \mathcal{A}^{n_{1, 2}}_{m_{1, 2}}\right]
+ Z^{n_{1, 2}}_{m_{1, 2}}A^{-1}_{r_{1}}+\left[\mathbb{B}_s, \mathcal{Y}^{n_{1, 2},{r_1}}_{m_{1, 2}}\right]&= &
\Le(\frac{\zeta^{{r_1},s, n_1}_{{m_1},{n_1}}}{a^{r_1,s}_{{m_1},{n_1}}}- \frac{\zeta^{{r_1},s,n_2}_{{m_2},{n_2}}}{a^{{r_1},s}_{{m_2},{n_2}}}\Ri)A_{{r_1}+{m_1}+{n_1}s}^{-1}
\\&&+
c^{0, s}_{{m_3},{n_3}+2}\Le(\frac{\zeta^{{r_1},s,n_1}_{{m_1},{n_1}}}{a^{{r_1},s}_{{m_1},{n_1}}}- \frac{\zeta^{{r_1},s,n_2}_{{m_2},{n_2}}}{a^{{r_1},s}_{{m_2},{n_2}}}\Ri)B_{{r_1}+{m_1}+{n_1}s-s}^0.
\eas

\end{cor}

\bpr
Let \(\mathcal{Y}^{n_1,r}_{m_1}\) and \(\mathcal{Y}^{n_2,r}_{m_2}\) denote the state solutions given in Lemma \ref{A-1Acal}.  Then,

\bas
&&\left[A^{-1}_{r_{1}}, \frac{\mathcal{A}^{n_j}_{m_j}}{a^{{r_{1}},s}_{{m_j},{n_j}}}\right]+ \frac{m_j-n_j+1}{a^{{r_{1}},s}_{{m_j},{n_j}}}Z^{n_j}_{m_j}A^{-1}_{r_{1}}+\left[\mathbb{B}_s, \frac{\mathcal{Y}^{n_j,{r_{1}}}_{m_j}}{a^{{r_{1}},s}_{{m_j},{n_j}}}\right]
\\ &=&\frac{\zeta^{{r_{1}},s,n_j}_{{m_j},{n_j}}}{a^{{r_{1}},s}_{{m_j},{n_j}}}A_{{r_{1}}+{m_j}+{n_j}s}^{-1}
+ \frac{b^{{r_{1}},s}_{{m_j},{n_j}}c^{{r_{1}},s}_{{m_j},{n_j}}}{a^{{r_{1}},s}_{{m_j},{n_j}}} B_{{r_{1}}+{m_j}+({n_j}-1)s}^0.
\eas

\noindent Recall that \(\delta(\mathcal{Y}^{n_1,r}_{m_1})= \delta(\mathcal{Y}^{n_2,r}_{m_2})= r+m_1+n_1s-2s.\) Therefore, for any \(r\in \mathbb{N}\) define
\be
\mathcal{Y}^{n_{1, 2},r}_{m_{1, 2}}:= \frac{1}{a^{s}_{{m_1},{n_1}}} \mathcal{Y}^{n_1,r}_{m_1}- \frac{1}{a^{s}_{{m_2},{n_2}}}\mathcal{Y}^{n_2,r}_{m_2}.
\ee Now, we claim that the condition \(m_3+n_3s=r+m_1+n_1s-2s\) implies
\bas
c^{0, s}_{{m_3},{n_3}+2}&=& \frac{\frac{b^{r,s}_{{m_1},{n_1}}c^{r,s}_{{m_1},{n_1}}}{a^{s}_{{m_1},{n_1}}}- \frac{b^{r,s}_{{m_2},{n_2}}c^{r,s}_{{m_2},{n_2}}}{a^{s}_{{m_2},{n_2}}}}{\frac{\zeta^{r,s,n_1}_{{m_1},{n_1}}}{a^{s}_{{m_1},{n_1}}}- \frac{\zeta^{r,s,n_2}_{{m_2},{n_2}}}{a^{s}_{{m_2},{n_2}}}}.
\eas
Assume that \(n_2=0\) and \(n_3=0.\) Then, \(m_1=m_2-n_1s,\) \(m_3=r+m_2-2s,\) \(\zeta^{r,s,0}_{m_2,0}=r+m_2+2,\) and \(b^{r,s}_{m_2,0}=0.\) Then,
\ba\label{comp1}
\frac{\frac{b^{r,s}_{m_1,n_1}c^{r,s}_{m_1,n_1}}{a^{s}_{m_1,n_1}}- \frac{b^{r,s}_{m_2,0}c^{r,s}_{m_2,0}}{a^{s}_{m_2,0}}}{
\frac{\zeta^{r,s,n_1}_{m_1,n_1}}{a^{s}_{m_1,n_1}}-\frac{\zeta^{r,s,0}_{m_2,0}}{a^{s}_{m_2,0}}
}&=&\frac{\frac{b^{r,s}_{m_2-n_1s,n_1}c^{r,s}_{m_2-n_1s,n_1}}{a^{s}_{m_2-n_1s,n_1}}- \frac{b^{r,s}_{m_2,0}c^{r,s}_{m_2,0}}{a^{s}_{m_2,0}}}{
\frac{\zeta^{r,s,n_1}_{m_2-n_1s,n_1}}{a^{s}_{m_2-n_1s,n_1}}-\frac{\zeta^{r,s,0}_{m_2,0}}{a^{s}_{m_2,0}}
}
\\\nonumber&=&
\frac{a^{s}_{m_2,0}b^{r,s}_{m_2-n_1s,n_1}c^{r,s}_{m_2-n_1s,n_1}}{a^{s}_{m_2,0}
\zeta^{r,s,n_1}_{m_2-n_1s,n_1}-(r+m_2+2)
a^{s}_{m_2-n_1s,n_1}}
\\&=&\label{comp2}\frac{(s+2)(r+m_2-s+1)}{(r+m_2-2s)(r+m_2-s)}= c^{0,s}_{m_3,2},
\ea where the last equation follows \(c^{r,s}_{m_2,0}= c^{r,s}_{m_2-n_1s,n_1}= c^{0,s}_{r+m_2-2s,2}= c^{0,s}_{m_3,2}.\) For arbitrary number \(n_2,\) assume that \(n_3=0, \) \ie \(m_3+2s= r+m_1+n_1s= r+m_2+n_2s.\) Therefore, Equations (\ref{comp1}--\ref{comp2}) imply
\bas
c^{0,s}_{m_3,2}&=&
\frac{\frac{b^{r,s}_{{m_1}-{n_1}s,{n_1}}c^{r,s}_{{m_1}-{n_1}s,{n_1}}}{a^{s}_{{m_1}-{n_1}s,{n_1}}}- \frac{b^{r,s}_{{m_1},0}c^{r,s}_{{m_1},0}}{a^{s}_{{m_1},0}}}{
\frac{\zeta^{r,s,n_1}_{{m_1}-{n_1}s,{n_1}}}{a^{s}_{{m_1}-{n_1}s,{n_1}}}-\frac{\zeta^{r,s,0}_{{m_1},0}}{a^{s}_{{m_1},0}}}
= \frac{\frac{b^{r,s}_{{m_2}-{n_2}s,{n_2}}c^{r,s}_{{m_2}-{n_2}s,{n_2}}}{a^{s}_{{m_2}-{n_2}s,{n_2}}}- \frac{b^{r,s}_{{m_2},0}c^{r,s}_{{m_2},0}}{a^{s}_{{m_2},0}}}{
\frac{\zeta^{r,s,n_2}_{{m_2}-{n_2}s,{n_2}}}{a^{s}_{{m_2}-{n_2}s,{n_2}}}-\frac{\zeta^{r,s,0}_{{m_2},0}}{a^{s}_{{m_2},0}}}.
\eas (Recall that \(\frac{a}{b}=\frac{c}{d}\) implies \(\frac{a}{b}=\frac{a-c}{b-d}.)\) Thus,
\bas
c^{0, s}_{{m_3},2}&=& \frac{\frac{b^{r,s}_{{m_1},{n_1}}c^{r,s}_{{m_1},{n_1}}}{a^{s}_{{m_1},{n_1}}}- \frac{b^{r,s}_{{m_2},{n_2}}c^{r,s}_{{m_2},{n_2}}}{a^{s}_{{m_2},{n_2}}}}{\frac{\zeta^{r,s,n_1}_{{m_1},{n_1}}}{a^{s}_{{m_1},{n_1}}}- \frac{\zeta^{r,s,n_2}_{{m_2},{n_2}}}{a^{s}_{{m_2},{n_2}}}},
\eas where \(n_3=0.\) When \(n_3\neq 0,\) the proof is complete by \(c^{0,s}_{m,n+2}= c^{0,s}_{m+ns,2}\) for any \(n\) and \(m.\)
\epr

\begin{cor}\label{ker}
Let \(m_1+n_1s=m_2+n_2s, m_3+n_3s= r+m_1+n_1s-2s,\) and
\ba\label{A3}
\mathcal{A}^{n_{1, 2, 3}, r}_{m_{1, 2, 3}}&:=& \mathcal{Y}^{n_{1, 2},r}_{m_{1, 2}}- \frac{\frac{\zeta^{r,s,n_1}_{{m_1},{n_1}}}{a^{s}_{{m_1},{n_1}}}- \frac{\zeta^{r,s,n_2}_{{m_2},{n_2}}}{a^{s}_{{m_2},{n_2}}}}{a^{s}_{{m_3},{n_3}}}\mathcal{A}^{m_3}_{n_3},
\\\label{Z3}
Z^{n_{1, 2, 3},r}_{m_{1, 2, 3}}&:=& - \frac{\frac{\zeta^{r,s,n_1}_{{m_1},{n_1}}}{a^{s}_{{m_1},{n_1}}}- \frac{\zeta^{r,s,n_2}_{{m_2},{n_2}}}{a^{s}_{{m_2},{n_2}}}}{a^{s}_{{m_3},{n_3}}}Z^{m_3}_{n_3}.
\ea Then, for \(r=r_1\) we have
%\(\left(\mathcal{A}^{n_{1, 2, 3}, r_1}_{m_{1, 2, 3}}, Z^{n_{1, 2, 3}, {r_1}}_{m_{1, 2, 3}}\right)\in \ker d^{{r_1}+m_1+n_1s-s,{r_1}-s+1}.\)
\bes
[A^{-1}_{r_1},\mathcal{A}^{n_{1, 2}}_{m_{1, 2}}]+ Z^{n_{1, 2}}_{m_{1, 2}}A^{-1}_{r_1} +[\mathbb{B}_s,\mathcal{A}^{n_{1, 2, 3}, r_1}_{m_{1, 2, 3}}]+ Z^{n_{1, 2, 3},r_1}_{m_{1, 2, 3}}\mathbb{B}_s= 0.
\ees
\end{cor}
The above gives rise to the following corollary.
\begin{cor} \label{r1s1}
The \(({r_1}-s+1)\)-th level orbital normal form is given by
\bes
A^1_0+\beta_s B^0_s+ \alpha_{r_1} A^{-1}_{r_1}+ \sum \alpha_i A^{-1}_i,
\ees where \(\alpha_{{r_1}+s^2+s}=0\), \(\alpha_{k(s+1)+2s}=0\) for any natural number \(k\neq s\), and \(r_1\neq k(s+1)+2s\) for any \(k\neq s.\) In addition,
we have \(\alpha_{s+r_2+{s}^{2}}=0\) when \(r_1= s(s+1)+2s\) and \(r_2<2s^2+4s\).

\end{cor}

The rest of this section is devoted to proving that any further simplification in the \((r_2-s+1)\)-th level is not possible when \(r_2< 2{r_1}-2s\).

\begin{lem}\label{Ynm}
Assume that \(\mathcal{Y}^{n,r}_m\) is given by Lemma \ref{A-1Acal}. Then, there exist \(\theta^{r,s}_{m,n},\) \(\kappa^{r,s}_{m,n}\in \mathbb{R}\) and  \(\mathbb{Y}^{n, r}_m\in \LST\) such that
\bas
\Le[A^{-1}_r, \mathcal{Y}^{n,r}_m\Ri]+ \Le[\mathbb{B}_s, \mathbb{Y}^{n,r}_m\Ri]&=& \theta^{r,s}_{m,n}A^{-1}_{2r+m+ns-2s} + c^{2r,s}_{m,n-2}\Le(\theta^{r,s}_{m,n}+\kappa^{r,s}_{m,n}\Ri) B^{0}_{2r+m+ns-3s}.
\eas

\end{lem}
\bpr The formulas for \(\mathbb{Y}^{n,r}_m,\) \(\theta^{r,s}_{m,n}\) and \(\kappa^{r,s}_{m,n}\) are explicitly derived. However, they are not presented here since they are too long.
\epr

\begin{cor} \label{corycalnn} Assume that \(m_1+n_1s=m_2+n_2s.\) Then, there exists a \(\mathbb{Y}^{n_1,n_2, r}_{m_{1, 2}}\) such that
\ba\label{ycalydouble}
\!\!\!&\Le[A^{-1}_r, \mathcal{Y}^{n_{1, 2}, r}_{m_{1, 2}}\Ri]+\Le[\mathbb{B}_s, \mathbb{Y}^{n_{1, 2}, r}_{m_{1, 2}}\Ri]=\!\!&\!\Le(\frac{\theta^{r,s}_{m_1,n_1}}{a^{s}_{m_1,n_1}}-\frac{\theta^{r,s}_{m_2,n_2}}{a^{s}_{m_2,n_2}}
\Ri)A^{-1}_{2r+m_1+n_1s-2s}
\\\nonumber\!\!\!&\!\!&\! + c^{2r,s}_{m_1,n_1-2}\Le(\frac{\theta^{r,s}_{m_1,n_1}}{a^{s}_{m_1,n_1}}-\frac{\theta^{r,s}_{m_2,n_2}}{a^{s}_{m_2,n_2}}
+\frac{\kappa^{r,s}_{m_1,n_1}}{a^{s}_{m_1,n_1}}
-\frac{\kappa^{r,s}_{m_2,n_2}}{a^{s}_{m_2,n_2}}\Ri)B^{0}_{2r+m_1+n_1s-3s}.
\ea

\end{cor}
\bpr Let \(\mathbb{Y}^{n_{1, 2}, r}_{m_{1, 2}}:= \frac{1}{a^{s}_{{m_1},{n_1}}} \mathbb{Y}^{n_1, r}_{m_1}- \frac{1}{a^{s}_{{m_2},{n_2}}}\mathbb{Y}^{n_2, r}_{m_2}.\) Then, the proof follows from Lemma \ref{Ynm}.
\epr

We assume that there exists a \(\alpha_k\neq0\) for some \(k>r_1\) where \(\alpha_k\) stands for the \({(r_1-s+1)}\)-th level coefficients of \(A^{-1}_k\). Let \be\label{ri} r_2:=\min \{\alpha_k\,|\, k>r_1, \alpha_k\neq0\}.\ee

\begin{lem}\label{3rdlevl} Assume that \(r_2< 2{r_1}-2s\) and let \(m_1+n_1s=m_2+n_2s\) and \(m_3+n_3s= {r_1}+m_1+n_1s-2s.\) Then, there exists a state solution \(\mathcal{A}^{n_{1, 2, 3}, {r_{1,2}}}_{m_{1, 2, 3}}\) such that
\bas
&&
\pi_{r_2+m_1+n_1s-s}\left(\left[A^{-1}_{r_2},\mathcal{A}^{n_{1, 2}}_{m_{1, 2}}\right]+ Z^{n_{1, 2}}_{m_{1, 2}}A^{-1}_{r_2}
+\left[A^{-1}_{r_1},\mathcal{A}^{n_{1, 2, 3}, r_{1,2}}_{m_{1, 2, 3}}\right] +\left[\mathbb{B}_s,\mathcal{A}^{n_{1, 2, 3}, r_1}_{m_{1, 2, 3}}\right]
+ Z^{n_{1, 2, 3},r_1}_{m_{1, 2, 3}}\mathbb{B}_s\right)
%\\&&d^{{r_2}+m_1+n_1s-s, r_2-s+1} \Le(\mathcal{A}^{n_{1, 2, 3}, {r_{1,2}}}_{m_{1, 2, 3}}, Z^{n_{1, 2, 3}, {r_1}}_{m_{1, 2, 3}}\Ri)
\\
&=& \left(\frac{\zeta^{{r_2},s, n_1}_{{m_1},{n_1}}}{a^{{r_2},s}_{{m_1},{n_1}}}- \frac{\zeta^{{r_2},s,n_2}_{{m_2},{n_2}}}{a^{{r_2},s}_{{m_2},{n_2}}}\right)A_{{r_2}+{m_1}+{n_1}s}^{-1}+
\left(\frac{\zeta^{{r_2},s, n_1}_{{m_1},{n_1}}}{a^{{r_2},s}_{{m_1},{n_1}}}- \frac{\zeta^{{r_2},s,n_2}_{{m_2},{n_2}}}{a^{{r_2},s}_{{m_2},{n_2}}}\right)c^{0, s}_{{m_3},{n_3}+2}B_{{r_2}+{m_1}+{n_1}s-s}^0,\eas where \(\pi_{r_2+m_1+n_1s-s}\) denotes the projection on the \(\delta\)-homogeneous space of grade \(r_2+m_1+n_1s-s.\)
Besides, \(c^{{r_1},s}_{{m_3},{n_3}}=c^{2{r_1},s}_{m_1,n_1-2}=c^{0,s}_{{m_3},{n_3}+2}\).
\end{lem}
\bpr Define
\ba\label{Acal123}
\mathcal{A}^{n_{1, 2, 3}, r_{1,2}}_{m_{1, 2, 3}}&:=&  \mathcal{Y}^{n_{1, 2},{r_2}}_{m_{1, 2}}+\mathbb{Y}^{n_{1, 2}, {r_1}}_{m_{1, 2}}- \frac{\frac{\zeta^{{r_1},s,n_1}_{{m_1},{n_1}}}{a^{{r_1},s}_{{m_1},{n_1}}}- \frac{\zeta^{{r_1},s,n_2}_{{m_2},{n_2}}}{a^{{r_1},s}_{{m_2},{n_2}}}}{a^{{r_1},s}_{{m_3},{n_3}}}\mathcal{Y}^{n_3,{r_1}}_{m_3}.
\ea
Since \(r_2< 2r_1-2s,\) \(\delta\Le(\Le[A^{-1}_{r_2}, \mathcal{A}^{n_{1, 2}}_{m_{1, 2}}\Ri]\Ri)=r_2+m_1+n_1s-s,\) and
\bes \delta\Le(\Le[A^{-1}_{r_1}, \mathcal{Y}^{n_{1, 2},{r_1}}_{m_{1, 2}}\Ri]\Ri)= \delta\Le(\Le[\mathbb{B}_s, \mathbb{Y}^{n_{1, 2},r_1}_{m_{1, 2}}\Ri]\Ri)=2{r_1}+m_1+n_1s-3s,\ees
 the proof follows from Corollary \ref{ZnDirect}.
\epr

\begin{thm}
The generalized saddle-node case system of Bogdanov--Takens given by (\ref{eq1}) can be transformed into its orbital normal form
\bas
\dot{x}&=& xy^{s}+ \alpha_{r_1}y^{{r_1}+1}+ \sum_{i>{r_1}} \alpha_i y^{i+1},
\\
\dot{y}&=&-x+ y^{s+1}.
\eas Here, \(\alpha_{k(s+1)+2s}=0\) for any natural number \(k\neq s\), \(\alpha_{{r_1}+s^2+s}=0\), and \(r_1\neq k(s+1)+2s\) for any \(k\neq s.\) Furthermore, for \(r_1= s(s+1)+2s\) and \(r_2<2s^2+4s\) we have \(\alpha_{s+r_2+{s}^{2}}=0\).
\end{thm}

\bpr The proof readily follows from Lemma \ref{3rdlevl} and Corollary \ref{r1s1}.
\epr

\section{The simplest normal form}\label{sec4}

The formulas obtained in Section \ref{sec3} are enough (along with Lemma \ref{Bs2}) to readily reproduce the corresponding results already obtained by Baider and Sanders \cite{baidersanders}. Define \(s_1:= s\). Let \(A:= \LST.\) Hence, Equation (\ref{LnN}) is governed by
\bas
d^{n,N}\Le(S_{n-N+1}, \ldots, S_{n-s_1}\Ri)&:=&\sum^{N-1}_{k=s_1} [v_{k}, S_{n-k}], \;\;\; S_i\in \LST_i.
\eas Then, Lemmas \ref{AReduce}--\ref{BReduce} imply that \(v\) given by Equation (\ref{eq1}) can be transformed into the \((s+1)\)-th level normal form
\bas
v^{({s_1}+1)}= A^1_0+ \sum^\infty_{i=r_1} a_iA^{-1}_i+ \sum^\infty_{j={s_1}} b_jB^0_j,
\eas where \(b_{m+(m+1){s_1}}=0\) for \(m\neq {s_1}\) and \(b_{m+(m+2){s_1}}=0\) for any natural number \(m.\) Assume that there exists \(b_j\neq0\) for some \(j>s_1.\) Define \(s_2:= \min \{b_j\,|\, b_j\neq 0, j>{s_1} \}.\)
\begin{lem} \label{Bs2}
There exists a \(\mathcal{B}^{s_1}_{{s_1}, s_2}\) such that
\([B^0_{s_2}, \mathcal{B}^{s_1}_{s_1}]+ [\mathbb{B}_{s_1},\mathcal{B}^{s_1}_{{s_1}, s_2}]= \frac {(s_2)^{{s_1}}_{{s_1}}}{(s_2+1)^{s_1}_{{s_1}+1}}B^0_{{s_1}+{s_2}+{s_1}^2}.\)
\end{lem}
\bpr We have \([B^0_{s_2}, \mathcal{B}^{s_1}_{s_1}]= (s_2-{s_1}) B^{s_1}_{{s_1}+s_2}\). Lemma \ref{BReduce} completes the proof.
\epr

The following theorem restates the unique normal form obtained by Baider and Sanders \cite{baidersanders}.

\begin{thm}
Any system associated with \(v\) given by Equation (\ref{eq1}) with the condition \(2{s_1}<r_1\) can be transformed with a near-identity state transformation into the simplest normal form
\bas
\frac{dx}{dt}&=& \sum^\infty_{i=r_1} \alpha_iy^{i+1}+ \sum^\infty_{j={s_1}} \beta_jxy^{j},
\\
\frac{dy}{dt}&= &-x+\sum^\infty_{j={s_1}} \beta_jy^{j+1},
\eas where \(\beta_{m+(m+1){s_1}}=0\) for \(m\neq {s_1}\), \(\beta_{m+(m+2){s_1}}=0\) for any natural number \(m,\) and
\begin{itemize}
  \item \(\beta_{r_1+{s_1}^2}=0\) for \(r_1-{s_1}<s_2.\)  Furthermore, \(\alpha_{r_1+{s_1}+{s_1}^2}=0\) if \(r_1= m+(m+2){s_1}-{s_1}^2.\)
  \item \(\beta_{{s_1}+s_2+{s_1}^2}=0\) when \(r_1-{s_1}\geq s_2.\) Furthermore, \(\alpha_{m+(m+3){s_1}}=0\) for \(s_2= m+(m+2){s_1}-{s_1}-{s_1}^2.\)
\end{itemize}
\end{thm}
\bpr
The proof is straightforward by Equation (\ref{kerG}), Lemmas \ref{AReduce}, \ref{BReduce}, \ref{Bs2} and \ref{ZnmXs}.
\epr

\section{Parametric normal form}\label{sec5}

This section is devoted to the computation of parametric normal form for the generalized saddle-node case of Bogdanov--Takens singularity.
Parametric changes of state variables, parametric time rescaling and reparametrization are all needed for parametric normal form computation. For a detailed study of parametric normal forms and unfolding see \cite{GazorMokhtariInt,GazorMokhtariEul,GazorYuSpec,Murd98,MurdBook,Murd04}. Let \(\PM:= \mathbb{R}[[\mu]]^p\) denote the \(p\)-dimensional vector formal power series in terms of \(\mu=(\mu_1, \mu_2, \ldots, \mu_p),\) \(F:=\mathbb{R}[[\mu]],\) and \(A:= \LST\oplus\R\oplus \PM\); recall Equations \eqref{R} and \eqref{L}. We define a grading structure on the space \(\PM\) by the grading function
\bes\delta(\mu^\m):=(r_1+2)|\m|.\ees
Hence, for any \(S_i\in \LST_i, T_i\in \R_i,\) and \(P_i\in \PM,\) Equation (\ref{LnN}) is given by
\bas
d^{n, N+1}\Le(S_{n-N}, T_{n-N}, P_{n-N},\ldots, S_{n-s}, T_{n-s}, P_{n-s}\Ri)&:=&\sum^{N}_{k=s} \Le(\Le[v_{k}, S_{n-k}\Ri]+ T_kv_{n-k}+D_\mu (v_{n-k})P_k\Ri),
\eas where \(D_\mu\) denotes derivative with respect to \(\mu.\) We call a parametric vector field \(w(x, y, \mu)\) a parametric
deformation of \(v(x, y)\) given by Equation (\ref{eq1}), if
\be\label{eq51}
w(x, y, \0):= v(x, y), \ \ v(0, 0, \0)=0,
\ee and \(\mu:= (\mu_1, \mu_2, \ldots, \mu_p)\) for some \(p\in \N\).

\begin{lem}
By invertible parametric state transformations any parametric deformation of Equation (\ref{eq1}) can be transformed into
\be\label{Pclas}
v(x, y, \mu):=A^1_0a(\mu)+ \sum^\infty_{k=-1} a_k(\mu)A^{-1}_k+ \sum^\infty_{k=0} b_k(\mu)B^{0}_k,
\ee where \(a(0)= 1\) and \(a_{-1}(0)= a_0(0)= b_0(0)=0.\)
\end{lem}
\bpr
Lemmas \ref{AReduce}--\ref{BReduce} prove that the only nonlinear terms \(a_k(\mu)A^{-1}_k\) and \(b_k(\mu)B^{0}_k\) can stay in the classical normal forms. Following \cite{Murd09Mal}, by a primary shift (\(x\mapsto x-a^0_{-1}(\mu)),\) we can simplify terms of the form \(A^{0}_{-1}a^0_{-1}(\mu).\)
\epr
For the rest of this section assume that there exist \(k\) and \(l\) such that \(a_{k}(0), b_{l}(0)\neq 0.\) Then, define
\bes
r_1:= \min\big\{k\,|\, a_{k}(0)\neq 0, k\geq 1\big\} \hbox{ and } s:= \min \big\{l\,|\, b_l(0)\neq 0, l\geq 1\big\},
\ees and let \(2s<r_1.\) Note that \(r_1\) must be iteratively updated following (similar to) Remark \ref{Rem3.9}.
\begin{lem}\label{rs1p}
There exists invertible changes of variables and parametric time rescaling that transform \(v\) given by Equation (\ref{Pclas}) into the \((r_1-s+1)\)-th level parametric normal form
\ba
v^{(r_1-s+1)}= a(\mu) A^1_0+ \sum^{s-2}_{j=0} b_j(\mu) B^0_j+ B^0_s+ \alpha_{r_1} A^{-1}_{r_1}+ \sum^\infty_{i=-1} a_i(\mu) A^{-1}_i,\ea
where \(a_{{r_1}+s^2+s}(\mu)=0\) and \(a_{k(s+1)+2s}(\mu)=0\) for any \(k\neq s, k\geq 0, \mu\in \mathbb{R}^p.\) Furthermore, \(a_{i}(0)=0\) for all \(-1\leq i\leq r_1\) and \(b_j(0)=0\) for \(0\leq j\leq s-2\). Besides, \(a_{s+r_2+{s}^{2}}(\mu)=0\) when \(r_1= s(s+1)+2s\) and \(r_2<2s^2+4s\).
\end{lem}
\bpr
Let \(k=0\) in Equation (\ref{XsBkk}) and \(n=m=0\) in Lemma \ref{ZnmXs}. Then, the associated solutions can be used to simplify \(A_{2s}^{-1}\mu^\m\) and \(B_{s}^0\mu^\m\) for any \(\m\) with \(|\m|\geq 1.\) Furthermore since \(b_s(\0)\neq0,\) transformations \((x,y):=(\alpha X,\beta Y)\) and \(t:=\beta \tau\) (for \(\alpha= {\rm sign}(b_s(\0))\) and \(\beta:=\frac{1}{\sqrt[s+1]{b_s(\0)\alpha^s}}\)) change the coefficient \(b_s(\0)\) into \(1.\) Because \([A^{0}_{-1}, \mathbb{B}_s]= A^0_{s-1}-\frac{s(s+2)}{s+1}B^0_{s-1},\) (by a secondary shift; see \cite{Murd09Mal}) terms of the form \(b_{s-1}(\mu)B^0_{s-1}\) can be eliminated. The rest of the proof follows the grading structure and parametric versions of Lemmas \ref{ZnmXs}, \ref{rss}, and Corollary \ref{Bk}. \epr

An \(N\)-degree truncated \((r_1-s+1)\)-th level parametric normal form \(w_N(x, y, \mu)\) is given by
\ba
w_N := A^1_0a(\mu)+ \sum^{s-2}_{i=0} B^0_ib_i(\mu)+ B^0_s+ \alpha_{r_1} A^{-1}_{r_1}+ \sum^{N-1}_{j=-1} a_j(\mu) A^{-1}_j,\ea
for \(N>r_1\). Then, an \(N\)-degree truncated parametric deformation of \(v\) given by Equation (\ref{eq1}) is called a {\it non-degenerate perturbation} if
\be\label{rank}\rank \Le(\frac{\partial w_N}{\partial \mu}\Big|_{\mu=0}\Ri)= n+s+1.\ee Here, \(n\) and the (increasing) sequence \(j_1, j_2, \ldots, j_n\) are defined by
\bes \Le\{j\,|\, 1\leq j<N, j\neq r_1+s^2+s\Ri\}\setminus\Le\{k(s+1)+2s\,|\, k\neq s, k\geq 0\Ri\}= \{j_1, j_2, \ldots, j_n\},\ees when
\(r_1\neq s(s+1)+2s\). In this paper we also consider \(r_1:= s(s+1)+2s\) when \(r_2<2s^2+4s\). Then, \(n\) and the sequence \(j_i\) are instead derived by
\bes \Le\{j\,|\, 1\leq j<N, j\neq 2s^2+4s,j\neq s+r_2+{s}^{2}\Ri\}\setminus\Le\{k(s+1)+2s\,|\, k\neq s, k\geq 0\Ri\}= \{j_1, j_2, \ldots, j_n\}.\ees
The following theorem presents the main result of this paper.

\begin{thm}\label{Thm}
There exist invertible reparametrization, parametric time rescaling and changes of state variables such that any non-degenerate perturbation of the generalized saddle-node case of Equation (\ref{eq1}) can be transformed into \(N\)-degree truncated parametric normal form
\ba\nonumber
\dot{x}&= & \mu_{s}+\mu_{s+1}y+ \sum^n_{i=1} (\alpha_{j_i}+\mu_{i+s+1}) y^{j_i+1}+xy^s+\sum^{s-2}_{k=0} \mu_{k+1}xy^{k},
\\\label{PNF}
\dot{y}&=& -x+ y^{s+1}+\sum^{s-2}_{k=0} \mu_{k+1}y^{k+1},
\ea where \(\alpha_{j_i}=0\) for all \(j_i<r_1,\) and \(\alpha_{r_1}\neq 0.\) The differential system (\ref{PNF}) yields the Hamiltonian--Eulerian decomposition.
\end{thm}
\bpr
The proof is straightforward by Lemma \ref{rs1p} and rank condition (\ref{rank}).
\epr

\section{Examples }\label{sec6}

In this section we present some useful formulas for normal form computations of Bogdanov--Takens singularity. First we derive the coefficients associated with the classical normal forms in Eulerian--Hamiltonian decomposition. Then, explicit formulas of some parametric normal form coefficients for the case \((s, r_1)=(1, 3)\) and \((s, r_1)=(1, 4)\) are presented. Given our Maple program, computation of the remaining cases can be readily derived.
\begin{prop}\label{Prop6.1} Consider a differential system given by
\bas
\frac{dx}{dt}&=& \sum_{i+j\geq 2} a_{i,j} x^iy^j,
\\
\frac{dy}{dt}&=&-x+ \sum_{i+j\geq 2} b_{i, j} x^i y^j.
\eas Then, some coefficients of the classical normal form
\be\label{6.1}
\left\{
  \begin{array}{ll}
\dot{x}=\sum^\infty_{k=1} \tilde{a}_{k} y^{k+1}+ \sum^\infty_{k=1} \tilde{b}_{k}xy^{k},
\\
\dot{y}=-x+ \sum^\infty_{k=1} \tilde{b}_{k}y^{k+1} & \hbox{.}
  \end{array}
\right.
\ee
are as follows:
\bas
\tilde{a}_1&=&a_{{0,2}},
\\
\tilde{a}_2&=&a_{{0,3}}-b_{{1,1}}a_{{0,2}}-\frac{4}{9}{b_{{0,2}}}^{2}-\frac{1}{9}{a_{{1,1}}}^{2
}+\frac{5}{9}a_{{1,1}}b_{{0,2}},
\\
\tilde{a}_3&=&a_{{0,4}}-\frac{1}{3}b_{{0,2}}a_{{1,1}}b_{{1,1}}-\frac{3}{2}b_{{1,1}}a_{{0,3}}-\frac{1}{2}a_{{2,0}}a_{{0,3}}+\frac{2}{3}
{b_{{0,2}}}^{2}a_{{2,0}}-\frac{1}{2}b_{{0,2}}a_{{1,1}}a_{{2,0}}
+\frac{1}{18}a_{{1,1}}b_{{2,0}}a_{{0,2}}
\\&&
-{\frac {14}{9}}b_{{0,2}}b_{{2,0}}a_{{0,2}}-\frac{2}{3}b_{{1,2}}a_{{0,2}}+\frac{1}{6}a_{{2,1}}a_{{0,2}}-\frac{1}{12}{a_{{2,0}}}^{2}a_{{0,2}}
+{\frac{7}{12}}{b_{{1,1}}}^{2}a_{{0,2}}-\frac{1}{6}a_{{1,1}}a_{{1,2}}+\frac{2}{3}b_{{0,2}}a_{{1,2}}
\\&&
+\frac{1}{2}a_{{2,0}}b_{{1,1}}a_{{0,2}}+\frac{1}{12}{a_{{1,1}}}^{2}b_{{1,1}}+\frac{1}{2}a_{{1,1}}b_{{0,3}}-b_{{0,2}}b_{{0,3}
}+\frac{1}{12}{a_{{1,1}}}^{2}a_{{2,0}},
\eas and
\bas
\tilde{b}_{1}&=&\frac{1}{3}a_{{1,1}}+\frac{2}{3}b_{{0,2}},
\\
\tilde{b}_2&=&\frac{1}{4}a_{{1,2}}+\frac{3}{4}b_{{0,3}}
-\frac{1}{8}a_{{1,1}}b_{{1,1}}-\frac{1}{8}a_{{1,1}}a_{{2,0}}+\frac{1}{4}b_{{0,2}}a_{{2,0}},
\\
\tilde{b}_3&=&\frac{1}{5}a_{{1,3}}+\frac{4}{5}b_{{0,4}}-{\frac{8}{45}}{b_{{0,2}}}^{2}b_{{2,0}}+\frac{1}{30}
{b_{{1,1}}}^{2}a_{{1,1}}+\frac{1}{15}b_{{0,2}}{b_{{1,1}}}^{2}
+\frac{1}{15}a_{{1,1}}{a_{{2,0}}}^{2}-{\frac {4}{15}}b_{{0,2}}{a_{{2,0}}}^{2}
\\&&-\frac{1}{9}b_{{0,2}}a_{{1,1}}b_{{2,0}}
-{\frac {1}{90}}{a_{{1,1}}}^{2}b_{{2,0}}+\frac{1}{10}a_{{2,0}}a_{{1,1}}b_{{1,1}}
-\frac{2}{5}b_{{0,3}}b_{{1,1}}-\frac{1}{5}b_{{0,3}}a_{{2,0}}
-\frac{1}{5}a_{{1,2}}b_{{1,1}}-\frac{1}{5}a_{{1,2}}a_{{2,0}}
\\&&-\frac{1}{15}a_{{1,1}}b_{{1,2}}+{\frac {4}{15}}b_{{0,2}}b_{{1,2}}
-\frac{1}{30}a_{{1,1}}a_{{2,1}}+\frac{1}{3}b_{{0,2}}a_{{2,1}}.
\eas
\end{prop}
\bpr The proof follows a symbolic computation using Maple.
\epr

The next proposition provides the formulas for parametric normal forms of the case \(s=1\) and \(r_1=3\) and \(4.\)

\begin{prop}\label{6.2} Consider a differential system given by Equation \eqref{6.1}, and assume that \(\tilde{a}_1=\tilde{a}_2=0,\) \(\tilde{b}_1\neq 0,\) and \(\tilde{a}_3\tilde{b}_1\neq \frac{2}{9}\tilde{b}_2\). Let \(a_k=\tilde{a}_k\tilde{b}_1|\tilde{b}_1|^{\frac{k}{2}}\) and \(b_k=\tilde{b}_k|\tilde{b}_1|^{\frac{k+1}{2}}\) for \(k>1.\) Then, the \(8\)-degree truncated parametric normal form of any non-degenerate perturbation of the system is given by
\ba\nonumber
\dot{x}&= & xy+\mu_{1}+\mu_{2}y+ \mu_{3}y^2+ \mu_{4}y^3+ (\alpha_3+\mu_5) y^4+ (\alpha_4+\mu_6) y^5+ \alpha_7 y^8,
\\\label{PNF6}
\dot{y}&=& -x+ y^{2},
\ea where
\bas
\alpha_3&:=&a_{3}-\frac{2}{9}b_{2},
\\
\alpha_4&:=&a_4+\frac{53}{81}{b_2}^2-\frac{8}{3}b_2a_3-\frac{1}{2}b_3,
\eas and \(\alpha_7\) follows Equation \eqref{alpha7} in Appendix. When \(\tilde{a}_3\tilde{b}_1=\frac{2}{9}\tilde{b}_2\) and \(5{\tilde{a_3}}^2|\tilde{b}_1|+4\tilde{a}_4\tilde{b_1}\neq 2\tilde{b}_3\), the \(8\)-degree truncated parametric normal form of any non-degenerate perturbation of the system is

\ba\nonumber
\dot{x}&= & xy+\mu_{1}+\mu_{2}y+ \mu_{3}y^2+ \mu_{4}y^3+\mu_5 y^4+ (\hat{\alpha}_4+\mu_6) y^5+ (\hat{\alpha}_5+\mu_7) y^6+ \hat{\alpha}_7 y^8,
\\\label{PNF6D}
\dot{y}&=& -x+ y^{2},
\ea %\basA^1_0+B^0_1+ \hat{\alpha}_4A^{-1}_4+\hat{\alpha}_5 A^{-1}_5+\hat{\alpha}_7 A^{-1}_7,&&\eas
where \(\hat{\alpha}_7\) is given in Equation \eqref{alphatilde7} in Appendix and
\bas
\hat{\alpha}_4&:=&\frac{1}{4a_4 -2b_3 +5 {a_3} ^{2}}\Big({b_3} ^{2}-5{ { {a_3} ^{2}b_3 }}+4{{ {a_4}^{2}}
}+{\frac {25}{4}}{ { {a_3} ^{4}}}-4{ {b_3 a_4 }}+10{{ {a_3} ^{2}a_4 }}\Big)\neq 0,
\\
\hat{\alpha}_5&:=&\frac{1}{4a_4 -2b_3 +5 {a_3} ^{2}}\bigg({\frac {136}{3}}{{a_3 b_3 a_4 }}+{\frac {712}{5}}{{ {a_3} ^{3}
a_4 }}-{\frac {16}{5}}{{b_4 a_4 }}-4{{ {a_3}^{2}b_4 }}
-2{{a_5 b_3 }}-{\frac {1268}{15}}{{ {a_3} ^{3}b_3 }}
\\&&-{\frac {26}{3}}{{a_3{b_3}^{2}}}+5{{ {a_3} ^{2}a_5 }}
+{\frac {531}{2}}{{ {a_3} ^{5}}}+\frac{8}{5}{{b_4 b_3 }}-56{{a_3{a_4}^{2}}}+4{{a_5 a_4 }}\bigg).
\eas
Furthermore, assume \(\hat{\alpha}_5\neq 0\). Then, \(\hat{\alpha}_7\) vanishes via applying an appropriate transformation.
\end{prop}

\bpr Given the rescalings \(\tilde{a}_k\) and \(\tilde{b}_k,\) Equation (\ref{6.1}) yields
\ba\label{1stlevel}
\!A_0^1+B_1^0+b_2B_2^0+a_3A_3^{-1}\!+b_3B_3^0+a_4A_4^{-1}\!+b_4B_4^0+a_5A_5^{-1}\!+b_5B_5^{0}+a_6A_6^{-1}\!+b_6B_6^0+a_7A_7^{-1}\!.
\ea The inequality \(\tilde{a}_3\tilde{b}_1\neq \frac{2}{9}\tilde{b}_2\) ensures that \(a_{3}\neq \frac{2b_{2}}{9}\). By Remark \ref{updatedr1} we have  \be r_1=3.\ee Since \(s=1\), we have \(2s<r_1.\)
Next, Theorem \ref{Thm} and symbolic implementation of the results gives rise to Equation \eqref{PNF6}.
%\ba\label{SNFClass}&&A_0^{1}+B_1^{0}+ \alpha_3 A_3^{-1}+ \alpha_4A_4^{-1}+ \alpha_7 A_7^{-1}.\ea

Now, assume that \(\tilde{a}_3\tilde{b}_1=\frac{2}{9}\tilde{b}_2\).
%Thus, the classical normal form is given by
%\besA_0^1+ B_1^0 + a_3 A_3^{-1} +\frac{9}{2} a_3 B_2^{0}+ a_4 A_4^{-1} +b_3 B_3^0 + a_5 A_5^{-1}
%+ b_4B_4^0 + a_6 A_6^{-1}+ b_5 B_5^0 + b_6 B_6^0 + a_7 A_7^{-1}\ees
This condition leads to elimination of both \(A^{-1}_3\) and \(B^0_2\). Furthermore, \(B^0_3\) can be eliminated by \(\mathfrak{B}^1_2\); see Lemma \ref{BReduce}. Then, the coefficient of \(A^{-1}_4\) is
\bes
a_4-\frac{1}{2}b_3+\frac{5}{4}{a_3}^{2},
\ees and \(5{\tilde{a_3}}^2|\tilde{b}_1|+4\tilde{a}_4\tilde{b_1}\neq 2\tilde{b}_3\) implies that \(4a_4 +5 {a_3} ^{2}\neq 2b_3\). Equation \eqref{r1} infers that \(r_1:=4\) and \(\hat{\alpha}_4\neq 0\). Symbolic computation and Theorem \ref{Thm} via Equation \eqref{PNF} conclude our claims.

Assume that \(\hat{\alpha}_5\neq 0.\) Since \(r_1:= s(s+1)+2s=4\), then \(r_2:=5< 2r_1-2s=6\). Thus by Remark \ref{Rem3.9}, \(A_{7}^{-1}\) is also simplified from the system.
Symbolic computation and Theorem \ref{Thm} via Equation \eqref{PNF} conclude our claims.
\epr
This paper concludes with an example to show that simple conditions can ensure that a parametric system can be reduced to \eqref{PNF6D}.
\begin{exm}
Consider the system
\ba\nonumber
\dot{x}&=&ax^2+ bxy,
\\\label{exm}
\dot{y}&=&-x+ cx^2+d xy+ by^2,
\ea
for \(b\neq 0.\) %\({b}^{2}(a-d)\neq \frac{1}{9} (a-d)\)
Then,
\bas
& \tilde{a}_3=\frac{1}{4}{b}^{2}(a-d), \qquad \tilde{a}_4= b^2\left(\frac{11}{320}d^{2}-{\frac {117}{320}}a^{2}+{\frac {53}{160}}ad-\frac{2}{5}{b}c \right), &
\\& \tilde{b}_1=b, \qquad \tilde{b}_2=\frac{1}{8}b(a-d), \qquad \tilde{b}_3=\frac{1}{10} b(d^{2}-2a^{2}-3cb+ad). &
\eas
Now \(s=1\) and
\bas
\tilde{a}_3\tilde{b}_1-\frac{2}{9}\tilde{b}_2%&=& \frac{1}{4}{b}^{3}(a-d)-\frac{1}{36}b(a-d)\\
&=& \frac{1}{4}b(a-d)\Le(b^2-\frac{1}{9}\Ri).\eas
For \(a\neq d\) and \(b\neq \pm\frac{1}{3}\), \(r_1=3.\) Hence, by Proposition \ref{6.2}, the \(8\)-degree truncated parametric normal form for any non-degenerate perturbation of (\ref{exm}) can be obtained by formulas \eqref{PNF6}.

Assume that \(a= d\), \(c\neq 0,\) and \(b\neq \pm\frac{\sqrt{3}}{2\sqrt{2}}\). Then, we have \(\tilde{a}_3=0,\)
\bas
5{\tilde{a}_3}^2|\tilde{b}_1|+4\tilde{a}_4\tilde{b}_1- 2\tilde{b}_3 = \frac{1}{5}b^2c\Le(3-8b^2\Ri)\neq 0, \hbox{ and }  r_1=4.
\eas Hence, the hypothesis for the parametric normal form \eqref{PNF6D} in Proposition \ref{6.2} is satisfied.

Finally for \(b=\pm\frac{1}{3}\), we have
\ba\label{6.4}
5{\tilde{a_3}}^2|\tilde{b}_1|+4\tilde{a}_4\tilde{b_1}- 2\tilde{b}_3
%\\&=& \frac{5}{3888}(a-d)^2\pm\frac{4}{27}\left(\frac{11}{320}d^{2}-{\frac {117}{320}}a^{2}
%+{\frac {53}{160}}ad\mp\frac{2}{15}c \right)\mp\frac{1}{15} (d^{2}-2a^{2}+ad\mp c)\\
&=&\left\{
       \begin{array}{ll}
\frac {391}{4860}{a}^{2}-\frac {49}{2430}ad-\frac {293}{4860}
{d}^{2}+\frac {19}{405}c, & \hbox{ if } \quad b=  \frac{1}{3},\\
-\frac {757}{9720}a^2+\frac {73}{4860}ad+\frac {611}{9720}d^2+\frac {19}{405}c
, & \hbox{ if } \quad\;\;\quad b= -\frac{1}{3}. \end{array}\right.
\ea Any non-degenerate perturbation of the system \eqref{exm} is reduced to \eqref{PNF6D} if the number given in formula \eqref{6.4} is nonzero.
\end{exm}

\section{Acknowledgment}

The first author acknowledges Professor James Murdock's generous helps, discussions and remarks (too numerous to individually mention) throughout past several years that directly and indirectly have contributed to this work. The first author hereby would like to thank Professor Henryk \.{Z}oladek for his financial support, hospitality and informative discussions during the first author's summer 2009 visit to the University of Warsaw. Finally, we thank the anonymous referees for their valuable remarks.

% ----------------------------------------------------------------

\section{Appendix}
In this appendix we present some formulas used in the paper.
\ba\nonumber
\zeta^{r,s, i}_{m,n}&:=&\frac{(r+m+2)(r+m)^{i}_{s}}{(r+m-n+3)^{i}_{s+1}}
-\sum_{l=1}^{i}\frac{(m)^{l-1}_{s}(r+m+ls)^{i-l}_{s}}{\big(r+m-n+3+l(s+1)\big)^{i-l}_{s+1}(m+2-n+s)^{l}_{s+1}}
\\\label{zeta}&&\times\Big({l}^{2}s-nls-2ls+lm+lr+s+ns+rs-nm+r-m-nr\Big).
\ea
\ba\nonumber
\alpha_7\!\!\!\!&:=\!&\frac{1}{2b_2 -9a_3}\bigg({\frac {837}{7}}{ {a_4 b_4 a_3 }{}}-{\frac {282}{7}}{ {b_3 b_4 a_3 }{}}+{\frac {579}{14}}{{b_3 a_5 a_3}{}}-{\frac {99}{35}}{ {b_2 a_6 a_3 }{}}+{\frac{44}{7}}{ {b_2 b_3 a_5 }{}}+{\frac {225}{28}}{ { {a_5}^{2}}{}}
+{\frac {36}{7}}{ { {b_4} ^{2}}{}}
\\\nonumber&&+{\frac {6502276}{31000725}}{{{b_2} ^{6}}{}}-{\frac {3665}{12}}{ {b_2 b_3 a_4 a_3 }{}}+{ {2b_2 a_7 }{}}+{\frac {102387959}{3444525}}{ { {b_2} ^{5}a_3 }{}}-{\frac {20}{7}}{ {b_2 b_6 }{}}
-{\frac {3676}{567}}{ { {b_2} ^{3}b_4 }{}}+{\frac{5989}{810}}{ { {b_2} ^{3}a_5 }{}}
\\\nonumber&&-{\frac {129}{35}}{ { {b_2} ^{2}a_6 }{}}
+{\frac {24}{7}}{ { {b_2} ^{2}b_5 }{}}+{\frac {22}{9}}{ { {b_2} ^{2} {a_4} ^{2}}{ }}-{\frac {71983}{15120}
}{ { {b_2} ^{2} {b_3} ^{2}}{}}+{\frac {3831413}{306180}}{ { {b_2} ^{4}b_3 }{}}-{ {9a_7 a_3 }{}}
-{\frac {982462}{25515}}{ { {b_2} ^{4}a_4 }{ }}
\\\nonumber&&-{\frac {177040867}{765450}}{ { {b_2} ^{4} {a_3} ^{2}}{}}-{\frac {13137}{35}}{ { {b_2} ^{2} {a_3} ^{4}}{ }}+{\frac {236822}{315}}{ { {b_2} ^{3} {a_3} ^{3}}{}}+{\frac {90}{7}}{ {b_6 a_3 }{}}
-{\frac {8154}{35}}{ {b_4  {a_3} ^{3}}{}}
+{\frac {1359}{8}}{ {b_3  {a_3}^{4}}{}}
\\\nonumber&&-{\frac {90}{7}}{ {a_5 b_4 }{}}+{\frac {1253}{320}}{ { {b_3} ^{2} {a_3} ^{2}}{}}+{\frac {12231}{140}}{ {a_6  {a_3} ^{2}}{}}-{\frac {4077}{70}}{ {a_5  {a_3} ^{3}}{}}-{\frac {459}{7}}{ {b_5  {a_3} ^{2}
}{}}-{\frac {4257}{70}}{ {b_2 b_4  {a_3} ^{2}}{}}+{\frac {2397}{10}}{ {b_2 b_3 {a_3} ^{3}}{ }}
\\\nonumber&&+{\frac {90863}{1680}}{ {b_2  {b_3} ^{2}a_3 }{}}
+{{339b_2  {a_4}^{2}a_3 }{}}+{\frac {94}{7}}{ {b_2 a_4 b_4 }{}}-{\frac {24}{7}}{ {b_2 b_3 b_4 }{ }}
+{\frac {906}{5}}{ {b_2 a_4  {a_3} ^{3}}{}}+{\frac {657}{35}}{ {b_2 a_5  {a_3} ^{2}}{}}
-\frac{6}{7}{ {b_2 b_5 a_3 }{}}
\\\nonumber&&
-{\frac {361}{14}}{ {b_2a_5 a_4 }{ }}+{\frac {147286}{567}}{ { {b_2} ^{3}a_4 a_3 }{}}-{\frac{3392327}{34020}}{ { {b_2} ^{3}b_3 a_3 }{ }}
-{\frac {71363}{70}}{ { {b_2} ^{2}a_4  {a_3} ^{2}}{}}+{\frac {1909}{252}}{ { {b_2} ^{2}a_5 a_3 }{ }}
\\\label{alpha7}&&+{\frac {4639}{252}}{ { {b_2} ^{2}b_3 a_4 }{}}+{\frac {6451}{315}}{ { {b_2} ^{2}b_4 a_3 }{}}+{\frac {4782209}{
15120}}{ { {b_2} ^{2}b_3  {a_3} ^{2}}{}}-{\frac {3051}{28}}{ {a_5 a_4 a_3 }{}}+{\frac {2979}{112}}{ {b_3 a_4  {a_3} ^{2}}{}}\bigg).
\ea
\ba\nonumber
\hat{\alpha}_7\!\!\!\!&:=&\frac{1}{4a_4-2b_3 +5 {a_3}^{2}}\bigg({\frac {24}{5}}a_3 a_6 b_3-52a_3 b_5a_4
+\frac{14}{3}a_3 b_5 b_3-\frac {3924}{25}a_3 b_4 a_5-\frac {224442}{35}{a_3}^3b_3 a_4
\\\nonumber&&-\frac {9938}{5}{a_3}^2a_5a_4-\frac {18548}{35}{a_3}^{2}b_4 b_3+\frac {12349}{20}{a_3}^2a_5 b_3 +\frac {263572}{175}{a_3}^{2}b_4 a_4
-6a_3 b_3  {a_4}^2+\frac {239}{5}a_3 {b_3}^{2}a_4
\\\nonumber&& +9a_5 b_3 a_4+\frac{16}{5}b_4 b_3 a_4 +\frac {144}{5}a_3 a_6 a_4-\frac {1030612}{35}{a_3}^{5}a_4 +\frac{60591}{10}{a_3}^{4}a_5
-\frac {36}{5}a_6 a_5 -\frac {50}{7}{a_3}^{2}b_6
\\\nonumber&&+\frac {2499}{5}{a_3} ^{3}b_5-\frac {72}{5}b_4  {a_4}^{2}+8500{a_3}^{3} {a_4}^{2}-\frac {76}{5}a_3  {b_3}^{3}-\frac {32}{5}b_5 b_4+\frac {20}{7}b_6 b_3
-\frac {11808}{25}{a_3}^{3}a_6 +4a_7 a_4
\\\nonumber&&-\frac{36}{5}a_5{b_3}^{2}+\frac{810141}{700}{a_3}^{3}{b_3}^{2}+\frac {144}{25}a_6 b_4 +5{a_3}^{2}a_7 -\frac {5246687}{1750}{a_3}^{4}b_4-\frac {40}{7}b_6 a_4
-2a_7 b_3 +8b_5 a_5
\\\label{alphatilde7}&&+\frac {104}{25}b_4  {b_3}^{2}+\frac {44397211}{4200}{a_3}^{5}b_3 +\frac {576}{5}a_3{a_5}^{2}+\frac {1296}{25}a_3  {b_4}^{2}-\frac {73881267}{3500}{a_3}^{7}\bigg).
\ea
\ba\nonumber
\!\frac{b^{r,s}_{m,n}c^{r,s}_{m,n}}{(s+2)}\!&\!=\!&\!
\sum_{l=0}^{n}\frac{(n-l+1)\big(r+2+m+ls\big)(m)^{l}_{s}}{\big(m+2+ls\big)(m+2-n+s)^{l}_{s+1}}
\bigg(\frac{\big(r+m+1+(n-1)s\big)(r+m+ls)^{n-l-2}_{s}}{
\big(r+m+3-n+l(s+1)\big)^{n-l}_{s+1}}
\\\nonumber&& -\frac{(r+m+ls)^{n-l-2}_{s}}{(r+m+ls+2)\big(r+m-n+2+(l+1)(s+1)\big)^{n-l-2}_{s+1}}
\bigg)
\\\nonumber\!\!\!&\!&\!
-\sum_{l=1}^{n}\frac{r(s+2)(m)^{l}_{s}}{(m+sl+2)(m+2-n+s)^{l-1}_{s+1}}
\bigg(\frac{\big(r+m+1+(n-1)s\big)\big(r+m+ls\big)^{n-l-2}_{s}}{
(r+m+3-n+ls+l)^{n-l}_{s+1}}
\\\nonumber&& -\frac{(r+m+ls)^{n-l-2}_{s}}{(r+m+ls+2)\big(r+m-n+2+(l+1)(s+1)\big)^{n-l-2}_{s+1}}
\bigg)
\\\nonumber&&
+\sum_{l=1}^{n-1}\frac{\big(n-l\big)(m)^{l-1}_{s}\big(r+m+(l-1)s\big)^{n-l-1}_{s}}{
(r+m+2+ls)(m+2-n+s)^{l-1}_{s+1}\big(r+m-n+2+l(s+1)\big)^{n-l-1}_{s+1}}
\\\nonumber&&+\frac{(m-n+1)(r+m+2)\big(r+m+1+(n-1)s\big)(r+m)^{n-2}_{s}}{(m+2)(r+m+3-n)^{n}_{s+1}}
\\\label{brs}&& -\frac{(m-n+1)(r+m)^{n-2}_{s}}{(m+2)\big(r+m+3-n+s\big)^{n-2}_{s+1}}.
\ea
\end{document}